\documentclass[11pt]{article}

\usepackage[margin=1in]{geometry}
\usepackage{amsmath,amssymb,amsthm,mathtools}
\usepackage{enumitem}
\usepackage[colorlinks=true,citecolor=blue,linkcolor=blue,urlcolor=blue]{hyperref}
\hypersetup{
  pdftitle={A Sharp Fuss-Catalan Bound for Norms of Powers of Random Matrices with i.i.d. Entries},
  pdfauthor={Yanjin Xiang, Kun Chen, and Zhihua Zhang},
  pdfkeywords={random matrices, powers of random matrices, operator norm, Fuss-Catalan numbers, high-moment method}
}

\newtheorem{theorem}{Theorem}[section]
\newtheorem{proposition}[theorem]{Proposition}
\newtheorem{lemma}[theorem]{Lemma}
\newtheorem{corollary}[theorem]{Corollary}
\newtheorem{definition}[theorem]{Definition}
\newtheorem{remark}[theorem]{Remark}
\newtheorem{example}[theorem]{Example}

\DeclareMathOperator{\Tr}{Tr}
\DeclareMathOperator{\Var}{Var}
\DeclareMathOperator{\E}{\mathbb E}
\DeclareMathOperator{\Prob}{\mathbb P}

\newcommand{\N}{\mathbb N}
\newcommand{\eps}{\varepsilon}
\newcommand{\ol}{\overline}
\newcommand{\wt}{\widetilde}
\newcommand{\one}{\mathbf 1}

\title{Universality for Products  of Random Matrices with i.i.d. Entries and the Fuss--Catalan Number}

\author{
Yanjin Xiang, Kun Chen, and Zhihua Zhang \\
School of Mathematical Sciences, Peking University\\
\texttt{2401110086@stu.pku.edu.cn}; \texttt{kchen0415@pku.edu.cn}; \texttt{zhzhang@math.pku.edu.cn}
}

\date{\today}
\begin{document}
\maketitle

\begin{abstract}
Let \((w_{ij})_{i,j\ge1}\) be a single infinite array of independent identically
distributed real- or complex-valued entries of mean zero, variance
\(\sigma^2\), and finite fourth moment.  Set
\(W_n=(w_{ij})_{1\le i,j\le n}\) and \(X_n=n^{-1/2}W_n\).  For every fixed
\(k\ge1\), we identify the almost sure limiting operator norm of several
fixed products built from this family.  Define the \(k\)-th freeness
coefficient by
\[
  \gamma_k:=\sqrt{\frac{(k+1)^{k+1}}{k^k}}.
\]
Then we prove
\[
  \|X_n^k\|\to\sigma^k\gamma_k
  \qquad \text{almost surely}.
\]
The same limit holds for products sampled with replacement from any fixed
finite pool of independent copies of \(X_n\); in particular, it holds for the
product of \(k\) independent copies. Thus, the freeness coefficient captures the
non-commuting characteristic between large random matrices 
under the finite fourth moment assumption.  
The improvement of the classical Bai--Yin-type power estimate
from the scale \(\sigma^k(k{+}1)\) to \(\sigma^k \sqrt{k{+}1}\) is a direct corollary
of our result.

The main technical challenge is to prove the upper bound using a high-moment expansion of
\(\E\Tr((X_n^kX_n^{*k})^m)\).  The leading zero-defect trace words are
tree-like and are counted by the Fuss--Catalan number
\[
  F_{k,m}=
  \frac1{km+1}\binom{(k+1)m}{m}.
\]
The combinatorial tool helps to devise a defect-sensitive global enumeration: if
\(L=km\) and
\[
  r=(L+1-v)+(L-q),
\]
then the number of admissible word classes with defect \(r\) is at most
\(F_{k,m}(Cm)^{Dr}\).  This polynomial-in-\(m\) loss, with degree proportional
to the defect, is summable in the logarithmic moment range.
\end{abstract}



\section{Introduction}

The spectral behavior of large random matrices is a central theme in modern
probability, data science, machine learning,  and high-dimensional statistics; see,
for example, the monograph of Bai and Silverstein \cite{bai-silverstein} and
the matrix concentration perspective developed in \cite{tropp-intro}.  Beyond
the empirical distribution of eigenvalues or singular values, the location of
the spectral edge and the operator norm are often the quantities that determine
stability, conditioning, and the performance of algorithms.  For sums of
independent random matrices, a large body of matrix concentration inequalities
provides flexible spectral-norm bounds, with more recent refinements using
higher-order information and free-probabilistic models to capture truly
noncommutative effects \cite{tropp-second-order,bandeira-boedihardjo-vanhandel,
bandeira-cipolloni-schroeder-vanhandel}.  Products are more delicate: the
order of the objects matters, the usual exponential-moment arguments for sums
do not directly apply, and even for products of independent matrices special
tools are needed \cite{huang-nilesweed-tropp-ward}.

For matrix powers, the difficulty is sharper still, because all objects are the same matrix and the trace expansion contains additional collision patterns that are absent in the
independent-product model.  Much of the sharp operator-norm technology is most
transparent in Gaussian or asymptotically free settings.  The purpose of our paper is to obtain the sharp spectral edge for powers and products under the substantially weaker assumption of a finite fourth
moment, by replacing the Gaussian/free-probabilistic argument with a direct
defect-sensitive high-moment enumeration.

Let \((w_{ij})_{i,j\ge1}\) be a single infinite array of real- or
complex-valued i.i.d. random variables satisfying
\[
  \E w_{11}=0,\qquad \E |w_{11}|^2=\sigma^2,\qquad
  \E |w_{11}|^4<\infty ,
\]
and set \(W_n=(w_{ij})_{1\le i,j\le n}\), \(X_n=n^{-1/2}W_n\).  We denote by
\(\|X_n\|\) the operator norm of \(X_n\).  We work with
this nested coupling throughout, so all almost sure limits are taken on the same
underlying probability space.  More generally,  for \(1\le \ell (\le k) \), let
\(Y_{1, n}, \ldots, Y_{\ell,n}\) be \(\ell\) independent copies of \(X_n\), and let
 \(X_{1, n}, \ldots, X_{k, n}\) be sampled from \(Y_{1, n}, \ldots, Y_{\ell,n}\) with
replacement. In this paper, we do not treat the
regime \(k=k(n)\), and study the universality of the
operator norm of products of the form
\[
  X_{1,n} \cdots X_{k, n},
\]
with particular emphasis on the sharp limiting upper bound.

There are two extreme cases: \(\ell=1\) and \(\ell=k\).  When \(\ell=1\), all
the objects are identical and the problem is to understand
\(\|X_n^k\|\).  When \(\ell=k\), the \(k\) objects are mutually independent.
One of the points of the paper is that these two extreme cases have the same
limiting constant, expressed in terms of freeness coefficients, which are defined via the Fuss--Catalan number.

\subsection{Fuss--Catalan numbers and free circular variables}
\label{sec:preliminaries}

We first give a brief introduction to the Fuss--Catalan numbers and the free-probability terminology.
For integers \(k,m\ge1\), the Fuss--Catalan number is
\begin{equation}
  F_{k,m}:=\frac1{km+1}\binom{(k+1)m}{m},
  \label{eq:Fkm}
\end{equation}
which also appears as a special case of the
Raney numbers; see Raney's power-series reversion formula
\cite{raney} and Stanley's account of Catalan-type enumeration
\cite{stanley-catalan}.  These numbers count, among many equivalent objects, rooted
ordered \((k+1)\)-ary trees with \(m\) internal vertices.  In the current
paper, the same number counts the zero-defect trace words in
\(\Tr((X^kX^{*k})^m)\) (see Sections~\ref{sec:leading-words} and \ref{sec:defect-counting}).

We define the \(k\)-th \emph{freeness} coefficient by
\begin{equation}
  \gamma_k:=
  \left(\lim_{m\to\infty}F_{k,m}^{1/m}\right)^{1/2}
  =
  \sqrt{\frac{(k+1)^{k+1}}{k^k}}.
  \label{eq:gammak}
\end{equation}
Equivalently, \(\gamma_k^2=(k+1)^{k+1}/k^k\).  The elementary inequality
\(\gamma_k\le\sqrt{e(k+1)}\) is responsible for the
\(\sqrt{k}\)-scale in the operator-norm estimates below.

In the free-probabilistic context, \(F_{k,m}\) is the \(m\)-th moment of the
\(k\)-fold free multiplicative convolution of the
Marchenko--Pastur law; see the standard references
\cite{voiculescu-dykema-nica,nica-speicher}.

We also recall the normalization of the circular element. 
In a tracial noncommutative probability
space \((\mathcal A,\tau)\), the state \(\tau\) is the free-probability
analogue of the normalized matrix trace \(n^{-1}\Tr\).  A circular element with
variance one may be written as
\[
  c=\frac{s_1+i s_2}{\sqrt2},
\]
where \(s_1,s_2\) are freely independent centered semicircular elements with
\[
  \tau(s_r)=0,\qquad \tau(s_r^2)=1,\qquad r=1,2.
\]
Then \(\tau(cc^*)=1\), and the free Wick formula gives
\[
  \tau\bigl((c^kc^{*k})^m\bigr)=F_{k,m}.
\]
Indeed, the contributing non-crossing pairings are exactly the balanced
tree-like pairings counted by \eqref{eq:Fkm}.  Hence
\(\|c^k\|=\gamma_k\), consistent with the computations of Oravecz and Larsen
\cite{oravecz,larsen}. This is a reason that we call $\gamma_k$ the \emph{freeness} coefficient.

\subsection{The Main Results}
\label{sec:main-results}

Throughout this paper \(k\) is fixed. Constants denoted by \(C,D,C_1,D_1\)
may depend on \(k\), and constants in Theorem~\ref{prop:moment} may also
depend on the chosen logarithmic moment parameter \(A\).  They never depend on
\(m\) or \(n\), unless stated explicitly.  We write
\(A^{*k}:=(A^*)^k\); in particular,
\(X_n^kX_n^{*k}=X_n^k(X_n^*)^k\).
We now state our main result in the following theorem. 

\begin{theorem}[Powers]
\label{thm:main}
For every fixed \(k\ge1\), we have
\[
  \limsup_{n\to\infty}\|X_n^k\|
  \le
  \sigma^k\gamma_k,
  \qquad \text{a.s.}
\]
and
\[
  \liminf_{n\to\infty}\|X_n^k\|
  \ge
  \sigma^k\gamma_k
  \qquad \text{a.s.}
\]
Consequently, we have
\[
  \|X_n^k\|\to
  \sigma^k\gamma_k
  \qquad \text{a.s.}
\]
\end{theorem}
By \eqref{eq:gammak}, the freeness coefficient
\(\gamma_k=\sqrt{(k+1)^{k+1}/k^k}\) is increasing in \(k\) and satisfies
\(\gamma_k\le\sqrt{e(k+1)}\). Thus, the upper bound above has the announced
\(\sqrt{k}\)-scale.
The Gaussian case admits an independent proof by free probability; we show it
as a corollary and give the proof in Section~\ref{sec:gaussian-proof}.

\begin{corollary}[Gaussian powers]
\label{cor:gaussian-powers}
Let \((g_{ij})_{i,j\ge1}\) be a single infinite array of i.i.d.\ standard
complex Gaussian entries, and set
\[
  G_n=(g_{ij})_{1\le i,j\le n},
  \qquad
  Y_n=n^{-1/2}G_n .
\]
Then, for every fixed \(k\ge1\),
\[
  \|Y_n^k\|\to \gamma_k
  \qquad \text{a.s.}
\]
\end{corollary}

Two special cases might be useful.  When \(k=1\), the constant is
\(\gamma_1=2\), recovering the Bai--Yin edge for \(\|X_n\|\).  When \(m=1\),
one has \(F_{k,1}=1\), corresponding to the normalization
\(n^{-1}\E\Tr(X_n^kX_n^{*k})\to\sigma^{2k}\).

We also present the corresponding statement for products whose objects are
sampled from a pool of independent matrix samples.
The next theorem tells the fixed-pool sampled-with-replacement version.  The
pool size is arbitrary but fixed; it does not enter the constant.

\begin{theorem}[Products sampled with replacement]
\label{thm:sampled-products}
Let \((X_{a,n})_{n\ge1}\), \(a\ge1\), be mutually independent copies of the
nested family \((X_n)_{n\ge1}\), with the same entry law.  Fix an arbitrary
integer \(N\ge1\), and let
\(I_1,\ldots,I_k\) be sampled with replacement from
\(\{1,\ldots,N\}\), independently of the matrices.  Then, for every fixed
\(k\ge1\),
\[
  \limsup_{n\to\infty}
  \|X_{I_1,n}X_{I_2,n}\cdots X_{I_k,n}\|
  \le
  \sigma^k\gamma_k
  \qquad \text{a.s.}
\]
and
\[
  \liminf_{n\to\infty}
  \|X_{I_1,n}X_{I_2,n}\cdots X_{I_k,n}\|
  \ge
  \sigma^k\gamma_k
  \qquad \text{a.s.}
\]
Consequently,
\[
  \|X_{I_1,n}X_{I_2,n}\cdots X_{I_k,n}\|
  \to
  \sigma^k\gamma_k
  \qquad \text{a.s.}
\]
The same conclusions hold if the index word is chosen deterministically with
\[
  (I_1,\ldots,I_k)\in\{1,\ldots,N\}^k .
\]
\end{theorem}

The independent-product case is the special case in which all sampled labels
are distinct.

\begin{corollary}[Independent products]
\label{cor:independent-products}
Let \(X_{1,n},\ldots,X_{k,n}\) be independent copies of \(X_n\), constructed
from independent nested arrays with mean zero, variance \(\sigma^2\), and
finite fourth moment.  Then, for every fixed \(k\ge1\),
\[
  \limsup_{n\to\infty}\|X_{1,n}X_{2,n}\cdots X_{k,n}\|
  \le
  \sigma^k\gamma_k
  \qquad \text{a.s.}
\]
and
\[
  \liminf_{n\to\infty}\|X_{1,n}X_{2,n}\cdots X_{k,n}\|
  \ge
  \sigma^k\gamma_k
  \qquad \text{a.s.}
\]
Consequently,
\[
  \|X_{1,n}X_{2,n}\cdots X_{k,n}\|
  \to
  \sigma^k\gamma_k
  \qquad \text{a.s.}
\]
\end{corollary}

Theorem~\ref{thm:main} and Corollary~\ref{cor:independent-products} show that
the \(k\)-th power \(X_n^k\) and the product of the \(k\) independent copies
\(X_{1,n}\cdots X_{k,n}\) have the same limiting norm:
\(\sigma^k\gamma_k\).  Thus, the same freeness coefficient appears both in the
independent-product model and in the non-Hermitian power.
Since \(\gamma_k=\|c^k\|\) for the circular element \(c\), this constant also
matches the one from the free-probability approach.

\subsection{Motivation and Related Work}

Recall that
the Bai--Yin high-moment method \cite{bai-yin}
implies that, for every fixed \(k\in\N\),
\begin{equation}
  \limsup_{n\to\infty}\|X_n^k\|\le (k {+} 1)\sigma^k
  \qquad \text{a.s.}
  \label{eq:bai-yin}
\end{equation}
The factor \(k{+}1\) is sufficient for bounding the spectral radius, because
\((k {+}1)^{1/k}\to1\).  It is not the desirable order for the operator norm of
\(X_n^k\).  Theorem~\ref{thm:main}  shows that, for fixed \(k\), the sharp
universal constant is of order \(\sqrt{k}\).

The appearance of Fuss--Catalan numbers in singular values of powers of
non-Hermitian random matrices is classical at the level of limiting empirical
distributions; see, for instance, Alexeev--G\"otze--Tikhomirov
\cite{alexeev-gotze-tikhomirov}, Alexeev's almost sure version
\cite{alexeev-almost-sure-fuss}, and the word setting studied by
Dubach--Peled \cite{dubach-peled}.  Such empirical-distribution results identify
the limiting moments and the right spectral edge of the squared singular values
of \(X_n^k\), but they do not by themselves rule out isolated singular values
above that edge under a finite fourth moment assumption.

Here we give a direct Bai--Yin-type high-moment upper bound under only a finite
fourth moment assumption, with the sharp fixed-\(k\) universal constant.  To the
best of our knowledge, this finite-fourth-moment operator-norm upper bound is not a direct
consequence of the existing empirical-distribution results.
One should also distinguish the present problem from products of independent
matrices.  For \(W_1\cdots W_k\) the factor labels suppress many collisions,
and the Fuss--Catalan moments arise from the independence between the factors.
Here the matrix is repeated, \(W^k\), so collisions between occurrences at
different positions of the same word survive and produce surplus edges and
non-tree identifications.  The defect counting in
Section~\ref{sec:defect-counting} is precisely the part of the argument that
controls these non-leading words at logarithmic moments.  The main
combinatorial point is that the non-leading words for \(W^k\), absent in the
independent-product model, can be controlled with only a polynomial defect loss
over the Fuss--Catalan leading count.

We emphasize that this is not merely a constant-level refinement of the
Bai--Yin estimate.  The classical Bai--Yin upper-bound argument organizes the
high-moment expansion through a coarse tree-growth count, which is well suited
to producing the \((k+1)\)-type bound in \eqref{eq:bai-yin}.  The present proof
uses a different combinatorial organization of the trace expansion.  It first
isolates the balanced tree-like words, whose exact count is the Fuss--Catalan number \(F_{k,m}\), and
then encodes the remaining repeated-matrix collisions by defect variables, pair
skeletons, finite core schemes, and local records.  The resulting estimate
\(F_{k,m}(Cm)^{Dr}\) for words of defect \(r\) keeps the Fuss--Catalan leading
count as the base and charges only a polynomial loss in the defect.  In this
sense, the sharp edge comes from a different enumeration principle, not from
optimizing constants in the earlier Bai--Yin count.

Free-probabilistic and sharp matrix-concentration methods provide powerful
tools for computing spectral edges of sums and products of large random
matrices, especially in Gaussian or asymptotically free models
\cite{haagerup-thorbjornsen,schultz,male,bandeira-boedihardjo-vanhandel,
bandeira-cipolloni-schroeder-vanhandel}.  The present proof is complementary:
it works directly with i.i.d.\ entries under only a finite fourth moment
assumption, and the non-Gaussian part of the argument is carried by the
truncation reduction and the defect enumeration of repeated-matrix trace words.

\subsection{Proof Roadmap and Paper Structure}

The roadmap of the proofs of our main results is as follows.  We first prove
Theorem~\ref{thm:main}.  We then prove Theorem~\ref{thm:sampled-products} by
adapting the same word expansion to products with sampled labels from a fixed
finite pool.  The
independent-product result, Corollary~\ref{cor:independent-products}, follows
as a special case of Theorem~\ref{thm:sampled-products}; the Gaussian result,
Corollary~\ref{cor:gaussian-powers}, is proved separately using strong
convergence and free probability.

The remainder of this paper is organized as follows. 
In Section~\ref{sec:proof-main-theorem}
we present the proof framework of Theorem~\ref{thm:main}.
In Section~\ref{sec:proof--main-proposition}
we prove the high-moment estimate, Theorem~\ref{prop:moment}, which is the key
building block for Theorem~\ref{thm:main}.  We give the proof of
the lower bound in Theorem~\ref{thm:main} in
Section~\ref{sec:lower-bound-proof}, followed by the proof of
Corollary~\ref{cor:gaussian-powers}.  The proofs of
Theorem~\ref{thm:sampled-products} and
Corollary~\ref{cor:independent-products} are given in
Section~\ref{sec:cor-proof}.
Section~\ref{sec:conclusion} concludes our work, and the remaining auxiliary
proofs are placed in the appendices.

\section{The Proof Framework of Theorem~\ref{thm:main}}
\label{sec:proof-main-theorem}

The proof proceeds by reducing to bounded triangular arrays and then removing
the truncation.

\subsection{Reduction to a High Moment Estimate}

If \(\sigma=0\), then \(w_{11}=0\) almost surely and the theorem is trivial.
Otherwise we first normalize \(\sigma=1\).  The general case follows by
replacing \(W_n\) with \(\sigma^{-1}W_n\).  The proof is based on the following
truncated high-moment estimate.

\begin{theorem}[High-moment estimate for truncated arrays]
\label{prop:moment}
Fix \(k\ge1\), \(A>0\), and let \(B>0\) be sufficiently large depending only on
\(k,A\).  Let \(X_n=n^{-1/2}W_n\), where for each \(n\) the matrix \(W_n\) has
independent entries \(w_{ij}^{(n)}\) satisfying
\[
  \E w_{ij}^{(n)}=0,\qquad
  \E |w_{ij}^{(n)}|^2=1,\qquad
  |w_{ij}^{(n)}|\le \eps_n\sqrt n,
\]
where the sequence \(\eps_n \ge 0\) satisfies \(\eps_n(\log n)^B\to0 \). The entries are not required to be identically distributed.
Then for every \(\eta>0\), uniformly for all integers \(1\le m\le A\log n\),
we have, for all sufficiently large \(n\),
\[
  \E\Tr\bigl((X_n^kX_n^{*k})^m\bigr)
  \le
  n(\gamma_k^2+\eta)^m .
\]
Here ``sufficiently large'' may depend on \(k,A,\eta\) and on the particular
sequence \((\eps_n)\), equivalently on the triangular array (see Lemma~\ref{lem:truncation}), but not on \(m\).
\end{theorem}

The upper bound in Theorem~\ref{thm:main} follows from this estimate and the
truncation lemma below. The proof framework of Theorem~\ref{prop:moment} 
is given in Section~\ref{sec:proof--main-proposition}. We  believe 
that this theorem is of independent interest.  

\subsection{Truncation}
\label{sec:truncation}

The required truncation reduction is stated in the nested form needed for
almost sure convergence.  The proof, together with the small-variance
Bai--Yin estimates used in it, is given in Appendix~\ref{app:reduction}.

\begin{lemma}[Truncation]
\label{lem:truncation}
Let \(W_n\) be the \(n\times n\) upper-left corner of a single infinite i.i.d.\
array with mean zero, variance one, and finite fourth moment. Then there exist a
deterministic sequence \(\eps_n\downarrow0\), with
\(\eps_n(\log n)^B\to0\) for every fixed \(B<\infty\), and a triangular
array \(\wt W_n=(\wt w_{ij}^{(n)})\) such that, for each fixed \(n\), the
entries \(\wt w_{ij}^{(n)}\), \(1\le i,j\le n\), are i.i.d. and satisfy:
\begin{enumerate}[label=(\roman*)]
  \item[\emph{(i)}] \(\E \wt w_{11}=0\) and \(\E|\wt w_{11}|^2=1\);
  \item[\emph{(ii)}] \(|\wt w_{ij}|\le \eps_n\sqrt n\) almost surely;
  \item[\emph{(iii)}]
  $
    \left\|n^{-1/2}(W_n-\wt W_n)\right\|\to0
    \qquad \text{a.s.}
  $
\end{enumerate}
The array \(\wt W_n\) is constructed entrywise from the same infinite array as
\(W_n\); only the truncation level, centering, and variance normalization
depend on \(n\).
\end{lemma}

\subsection{The Upper Bound in Theorem~\ref{thm:main}}
\label{sec:main-proof}

\begin{proof}[Proof of the upper bound in Theorem~\ref{thm:main}]
If \(\sigma=0\), then \(w_{11}=0\) almost surely and the theorem is trivial.
Assume first that \(\sigma=1\).  Let \(\wt W_n\) be the truncated array supplied
by Lemma~\ref{lem:truncation}, and put
\[
  \wt X_n=n^{-1/2}\wt W_n .
\]

We prove the bound first for the truncated matrices.  Fix
\(\delta,\eta>0\), and choose \(m=\lfloor A\log n\rfloor\).  Markov's
inequality and Theorem~\ref{prop:moment} give
\begin{align*}
  \Prob\left(\|\wt X_n^k\|>(1+\delta)\sqrt{\gamma_k^2+\eta}\right)
  &\le
  \frac{\E\Tr\bigl((\wt X_n^k\wt X_n^{*k})^m\bigr)}
       {(1+\delta)^{2m}(\gamma_k^2+\eta)^m}  \\
  &\le n(1+\delta)^{-2m},
\end{align*}
provided \(A\) is chosen and then the corresponding \(B\) in
Theorem~\ref{prop:moment} is used.  Because Lemma~\ref{lem:truncation} gives
\(\eps_n(\log n)^B\to0\) for every fixed \(B\), we apply this theorem to
\(\wt X_n\).  If \(A>3/(2\log(1+\delta))\), the last bound is summable.  By the
Borel--Cantelli theorem and then by applying the conclusion to countable sequences
\(\delta,\eta\downarrow0\),
\[
  \limsup_{n\to\infty}\|\wt X_n^k\|
  \le
  \gamma_k
  \qquad\text{a.s.}
\]
The same argument with \(k=1\) gives \(\|\wt X_n\|=O(1)\) almost surely.
Lemma~\ref{lem:truncation} also gives
\(\|X_n-\wt X_n\|\to0\), hence \(\|X_n\|=O(1)\) almost surely.  Therefore,
\[
  X_n^k-\wt X_n^k
  =
  \sum_{j=0}^{k-1}X_n^j(X_n-\wt X_n)\wt X_n^{k-1-j}
\]
implies
\[
  \|X_n^k-\wt X_n^k\|
  \le
  k(\|X_n\|+\|\wt X_n\|)^{k-1}\|X_n-\wt X_n\|
  =o(1)
  \qquad\text{a.s.}
\]
This proves the result when \(\sigma=1\).  If \(\sigma>0\), apply the normalized
case to \(\sigma^{-1}W_n\) and multiply the resulting bound by \(\sigma^k\).
\end{proof}

\section{The Proof Framework of Theorem~\ref{prop:moment}}
\label{sec:proof--main-proposition}

We prove the truncated high-moment estimate.  
It includes four main steps: 
expanding the trace
into closed words, identifying the Fuss--Catalan leading words, bounding all
non-leading words by a defect-sensitive counting estimate, and finally summing the
resulting bounds in the logarithmic moment range.
In the following proofs, \(n\) is fixed.  We write \(X=X_n\) and
\(w_{ij}=w_{ij}^{(n)}\), suppressing the triangular-array superscript.

\subsection{The Word Expansion}

Put \(L=km\).  Expanding the trace gives
\begin{equation}
\label{eq:trace-expansion}
  \E\Tr\bigl((X^kX^{*k})^m\bigr)
  =
  n^{-L}
  \sum_{\mathbf i}
  \E \prod_{s=1}^{2L} \zeta_s(\mathbf i),
\end{equation}
where \(\mathbf i=(i_0,i_1,\ldots,i_{2L})\), \(i_{2L}=i_0\), and $s$ represents the sign
pattern which is \(m\) repetitions of the block consisting of \(k\) plus signs
followed by \(k\) minus signs; and
\[
  \zeta_s(\mathbf i)=
  \begin{cases}
    w_{i_{s-1},i_s},& s\text{ is a plus occurrence},\\[2mm]
    \ol{w_{i_s,i_{s-1}}},& s\text{ is a minus occurrence}.
  \end{cases}
\]
We regard \(w_{uv}\) and \(\ol{w_{uv}}\) as belonging to the same ordered
variable-edge \((u,v)\).  In both the cases, the ordered variable-edge attached to the occurrence is the
ordered pair indexing the non-conjugated variable; hence it is
\((i_{s-1},i_s)\) at a plus occurrence and \((i_s,i_{s-1})\) at a minus
occurrence.  We denote this ordered pair by
\[
  e_s(\mathbf i)=
  \begin{cases}
    (i_{s-1},i_s),& s\text{ is a plus occurrence},\\[1mm]
    (i_s,i_{s-1}),& s\text{ is a minus occurrence}.
  \end{cases}
\]
For a closed word \(\mathbf i\), let
\[
  v(\mathbf i)=\#\{i_0, \ldots, i_{2L-1}\}
\]
be the number of visited vertices and let \(q(\mathbf i)\) be the number of
distinct ordered variable-edges among the \(e_s(\mathbf i)\)'s.

Two words are equivalent if one can be obtained from the other by relabeling
the visited vertices.  If a word has \(v\) vertices, its equivalence class has
at most \(n(n-1)\cdots(n-v+1)\le n^v\) representatives.

\begin{definition}
A word is \emph{admissible} if every ordered variable-edge appearing in the
word appears at least twice, namely if
\[
  r_e(\mathbf i):=\#\{s\colon e_s(\mathbf i)=e\}\ge2
\]
for every ordered variable-edge \(e\) that appears.  It is
\emph{balanced tree-like} if
\[
  q=L,\qquad v=L+1,
\]
and each ordered variable-edge appears exactly once at a plus sign and exactly
once at a minus sign.
\end{definition}

Only admissible words contribute to \eqref{eq:trace-expansion}, because
\(\E w_{11}=0\).  Indeed, if an ordered variable-edge appears only once, then
conditioning on all other matrix entries leaves a single centered factor
\(w_{uv}\) or \(\ol{w_{uv}}\), and the expectation vanishes by independence.
No assumption such as \(\E w_{11}^2=0\) is made in the complex case.  Same-sign
repetitions are allowed among admissible words; they are non-leading
collisions and are controlled by absolute moment bounds.  The leading
tree-like words are automatically plus/minus balanced.
Moreover, for an admissible word,
\[
  v\le q+1\le L+1,
\]
because the graph traced by the word is connected.

\begin{example}
For \(k=2\) and \(m=1\), the sign patterns are \(+,+,-,-\), and a closed word is
\(\mathbf i=(i_0,i_1,i_2,i_3,i_4)\) with \(i_4=i_0\).  The four factors are
\[
  w_{i_0,i_1}\,w_{i_1,i_2}\,
  \ol{w_{i_3,i_2}}\,\ol{w_{i_0,i_3}} .
\]
Thus the ordered variable-edges are
\[
  (i_0,i_1),\quad (i_1,i_2),\quad (i_3,i_2),\quad (i_0,i_3).
\]
The corners are the four positions \(0,1,2,3\), and a word-equivalence class
is obtained by recording which of these corners carry the same vertex label.
For instance, if the first and third ordered variable-edges are identified,
then the corner equalities are \(i_0=i_3\) and \(i_1=i_2\).  This is the
corner-partition viewpoint used below for pair skeletons.
\end{example}

\subsection{Fuss--Catalan Enumeration of the Leading Words}
\label{sec:leading-words}

Recall the Fuss--Catalan number \(F_{k,m}\) from \eqref{eq:Fkm}.  The leading
words are counted by this number.

\begin{lemma}[Leading words]
\label{lem:leading}
The number of equivalence classes of balanced tree-like words that contribute to
\(\Tr\bigl((X^kX^{*k})^m\bigr)\) is \(F_{k,m}\).
\end{lemma}

We will also use the elementary uniform bound
\begin{equation}
\label{eq:F-growth}
  F_{k,m}
  =
  \frac1{km+1}\binom{(k+1)m}{m}
  \le
  \left(\frac{(k+1)^{k+1}}{k^k}\right)^m
  =
  \gamma_k^{2m} .
\end{equation}
The proof of Lemma~\ref{lem:leading} and of \eqref{eq:F-growth} is given in
Appendix~\ref{app:combinatorics}.

\subsection{Defect Counting}
\label{sec:defect-counting}

The combinatorial estimate below replaces the coarse
\((k+1)\)-counting in \eqref{eq:bai-yin}.  For an admissible word, we define its
defect by
\[
  d(\mathbf i):=(L+1-v(\mathbf i))+(L-q(\mathbf i)).
\]
Then \(d=0\) if and only if the word is balanced tree-like; see
Appendix~\ref{app:combinatorics}.

\begin{definition}
Number the \(2L\) occurrences cyclically and also number the \(2L\) corners
between consecutive occurrences by \(0,\ldots,2L-1\).  For an occurrence \(s\),
write \(t(s)\) and \(h(s)\) for the tail and head of the ordered variable-edge:
\[
  (t(s),h(s))=
  \begin{cases}
    (s-1,s),& s\text{ is a plus occurrence},\\
    (s,s-1),& s\text{ is a minus occurrence},
  \end{cases}
\]
with indices read modulo \(2L\). Thus, a word-equivalence class is the same
thing as a partition \(\pi\) of the corner set, and the ordered variable-edge
at \(s\) is \((\pi(t(s)),\pi(h(s)))\).

A \emph{pair skeleton} is a pairing \({\mathcal P}\) of the \(2L\) occurrence
positions.  Let \(\pi({\mathcal P})\) be the finest partition of the corners
for which, whenever \(\{r,s\}\in{\mathcal P}\),
\[
  t(r)\sim t(s),\qquad h(r)\sim h(s).
\]
Put
\[
  v({\mathcal P})=|\pi({\mathcal P})|,
  \qquad
  \chi({\mathcal P})=L+1-v({\mathcal P}).
\]
The pairing is \emph{leading} if \(\chi({\mathcal P})=0\).
\end{definition}

The purpose of a pair skeleton is to forget multiplicities beyond two while
retaining enough endpoint identifications to control the loss of vertices.

We use the following dictionary.  Start with a rooted \(2L\)-gon whose sides
are the occurrence positions \(1,\ldots,2L\), read in the cyclic order
\(\gamma=(1\,2\,\cdots\,2L)\).  The side \(s\) is oriented from the corner
\(t(s)\) to the corner \(h(s)\).  A pair
\(\{r,s\}\in{\mathcal P}\) glues the side \(r\) to the side \(s\) by matching
tail to tail and head to head.  Therefore the vertices of the quotient polygon
are exactly the blocks of \(\pi({\mathcal P})\).  The quotient has one face,
\(L\) edges, and \(v({\mathcal P})\) vertices.  A pair joining opposite signs
is an untwisted band and a same-sign pair is a twisted band.  Conversely, any
labelled one-face gluing of this signed polygon gives such a pair skeleton.
Thus \(\chi({\mathcal P})=L+1-v({\mathcal P})\) is precisely the one-face
excess of the gluing, namely \(E-V+1\) for this one-face graph; it is not the
Euler characteristic of the surface.  The case \(\chi=0\) is the planar tree
case.
We use this quotient only as a signed one-face graph with a fixed boundary
tour; no orientability is assumed.  The cyclic order data used below are
inherited from the boundary tour and are recorded as part of the scheme.

\begin{lemma}[Core decomposition of pair skeletons]
\label{lem:opening}
There are constants \(C_1,D_1\), depending only on \(k\), such that the number
of pair skeletons with \(\chi({\mathcal P})=s\) is at most
\[
  F_{k,m}(C_1m)^{D_1s}.
\]
\end{lemma}

\begin{lemma}[Pairing surplus occurrences]
\label{lem:surplus}
Let \(\mathbf i\) be an admissible word with
\[
  a=L-q(\mathbf i),\qquad b=L+1-v(\mathbf i).
\]
Then \(\mathbf i\) can be encoded by a pair skeleton \({\mathcal P}\) with
\[
  \chi({\mathcal P})\le b+2a
\]
and by at most \(C(a+b)\) local records.  For fixed \({\mathcal P}\), the
number of possible records is at most \((Cm)^{D(a+b)}\).
\end{lemma}

\begin{proposition}[Refined word counting]
\label{prop:defect-counting}
There are constants \(C=C(k)\) and \(D=D(k)\) such that, for every \(m\ge1\)
and every \(r\ge0\), the number of equivalence classes of admissible words
with defect \(r\) is at most
\[
  F_{k,m}(Cm)^{Dr}.
\]
\end{proposition}

The detailed encoding estimates are proved in
Appendix~\ref{app:combinatorics}.  We include the main counting argument here
because Proposition~\ref{prop:defect-counting} is the combinatorial input for
the moment estimate.  The proof of Lemma~\ref{lem:opening} is a
graph-theoretic scheme decomposition.  In particular, the argument does not
rely on a local opening theorem for non-orientable unicellular maps.  The use
of a finite core plus tree expansions is in the spirit of classical
fixed-topology map decompositions, such as Bender--Canfield
\cite{bender-canfield}.

Given a pair skeleton \({\mathcal P}\) with \(s=\chi({\mathcal P})\), let
\(G({\mathcal P})\) be its connected one-face quotient graph.  Its
\emph{scheme} \(K({\mathcal P})\) is obtained by recursively pruning
degree-one branches, then suppressing maximal degree-two paths; if the
remaining two-core is a single cycle, it is replaced by one vertex with one
loop.  The rooted boundary tour gives the cyclic order data of the scheme.
After the single-cycle exception, every scheme vertex has degree at least
three, and therefore the number \(p\) of scheme half-edges and scheme corners
satisfies
\[
  p\le Cs .
\]
The suppressed degree-two paths are not discarded.  They are precisely the
edge-expansions of the scheme.
In the single-cycle case this loop is treated as a scheme edge, not as a
pruned tree; its full expansion is counted by the same two-rooted edge-envelope
mechanism.

\begin{proposition}[Expansion of a fixed scheme]
\label{prop:fixed-scheme-expansion}
Fix a rooted signed one-face scheme with \(p\) half-edges and corners, including
the finite endpoint/side convention attached to each scheme edge.  Once the
\(O(p)\) attachment positions in the cyclic word \((z^kz^{*k})^m\) are fixed,
the number of compatible expansions of all scheme edges, together with all
rooted tree branches attached to them and to the scheme vertices, is at most
\[
  F_{k,m}(Cm)^{Cp},
\]
where \(C=C(k)\).
\end{proposition}

\begin{proof}
We give the main encoding here; the injective bookkeeping is formalized in
Appendix~\ref{app:combinatorics}.  Cut the fixed cyclic word at the \(O(p)\)
scheme attachment positions.  Cuts
inside a block \(z^k\) or \(z^{*k}\) contribute only finite offset data, hence
only a factor \(C^p\), and they shift the number of complete blocks by \(O(p)\).
The rooted scheme, together with its boundary tour, assigns each complementary
arc either to a vertex sector or to one side of a specified scheme edge; hence
different scheme-edge expansions have no additional interleaving freedom.

The pruned components in vertex sectors are bridge trees, so their one-face
contour crosses every edge once in each direction.  They are therefore counted
by rooted leading tree-like branches.  A suppressed degree-two chain is treated
differently: its regular neighbourhood has two boundary strings
\[
  I=(i_1,\ldots,i_\ell),\qquad J=(j_1,\ldots,j_\ell),
\]
and a finite endpoint convention pairs \(i_r\) with either \(j_r\) or
\(j_{\ell+1-r}\).  This chain is encoded by a decorated two-rooted leading
envelope.  The envelope is only a counting device; it is not asserted that the
original chain is leading.  After the two boundary strings are placed back into
the fixed cyclic word, the same-sign or opposite-sign status of each restored
pair is read from the occurrence labels.  Thus a chain of length \(\ell\) does
not contribute a factor \(2^\ell\); see Lemma~\ref{lem:no-twist-record}.

Consequently, all sectors and scheme-edge envelopes form an ordered forest of
leading tree-like objects with \(O(p)\) roots and \(O(p)\) marked corners.  If
\(M(t)=1+tM(t)^{k+1}\) is the leading-tree generating function, such forests are
bounded by coefficients of \(M(t)^r\) with \(r=O(p)\), an \(O(p)\) shift in the
block count, and \(O(p)\) marked corners.  Lagrange inversion and adjacent
coefficient-ratio estimates compare all these coefficients to
\(F_{k,m}\) at the cost of \((Cm)^{Cp}\).  Lemmas
\ref{lem:edge-expansion-normal-form}--\ref{lem:scheme-expansion} give the
injective version of this encoding.
\end{proof}

We next combine the pieces.  Lemma~\ref{lem:surplus} is the
word-to-skeleton step: starting from an admissible word, choose two principal
occurrences for each ordered variable-edge and pair the remaining surplus
occurrences canonically.  If
\[
  a=L-q,\qquad b=L+1-v,
\]
then this produces a pair skeleton \({\mathcal P}\) with
\[
  \chi({\mathcal P})\le b+2a
\]
and only \((Cm)^{O(a+b)}\) possible local records.  Lemma~\ref{lem:opening}
follows by recording the \(O(s)\)-sized scheme of a skeleton with
\(\chi({\mathcal P})=s\), the \(O(s)\) attachment positions in the boundary
word, and then applying Proposition~\ref{prop:fixed-scheme-expansion}. Thus,
pair skeletons with excess \(s\) are counted by
\[
  F_{k,m}(Cm)^{O(s)}.
\]
Combining this with the local records and using \(s\le b+2a\le2(a+b)\) gives
Proposition~\ref{prop:defect-counting}.

\begin{remark}
The feature of Proposition~\ref{prop:defect-counting} is that the base count is
\(F_{k,m}\), not a crude Catalan number of length \(km\) and not a
local \((k+1)^{2m}\) factor.  The polynomial loss \((Cm)^{Dr}\) is harmless
when \(m=O(\log n)\).
\end{remark}

\subsection{To Complete the Moment Estimate}

We finally make sum over the defects.  For an admissible word, set
\[
  a=L-q,\qquad b=L+1-v .
\]
The bounded-entry moment estimate gives a factor
\(\eps_n^{2a}n^a\), while summing vertex labels contributes
\(n^{-L}n^v\).  Thus a class with parameters \((a,b)\) contributes at most
\[
  n^{-L}n^v\eps_n^{2a}n^a
  =
  n\,\eps_n^{2a}n^{a-b}.
\]
Since admissibility gives \(v\le q+1\), one has \(b\ge a\).  Writing
\(c=b-a\ge0\), and using Proposition~\ref{prop:defect-counting} for defect
\(a+b=2a+c\), the total contribution is bounded by
\[
  nF_{k,m}
  \sum_{a,c\ge0}
  \left((Cm)^{2D}\eps_n^2\right)^a
  \left((Cm)^Dn^{-1}\right)^c .
\]
After \(k,A\) and the constants in
Proposition~\ref{prop:defect-counting} have been fixed, choose \(B\) large
enough that \((CA\log n)^{2D}\eps_n^2\to0\).  Then, for \(m\le A\log n\), both
geometric ratios are \(o(1)\) uniformly in \(m\).  This is the only point where
the logarithmic moment range is used in an essential way.  Hence the moment is
\[
  \E\Tr\bigl((X_n^kX_n^{*k})^m\bigr)
  \le nF_{k,m}(1+o(1))
  \le n(\gamma_k^2+\eta)^m .
\]
The fully detailed proof of Theorem~\ref{prop:moment}, including the
uniformity in \(m\), is given in Appendix~\ref{app:moment}.

\section{The Lower Bound in Theorem~\ref{thm:main}}
\label{sec:lower-bound-proof}
\begin{proof}[Proof of the lower bound in Theorem~\ref{thm:main}]
If \(\sigma=0\), the claim is trivial.  Assume \(\sigma>0\).  Let
\(\mu_{n,k}\) be the empirical eigenvalue distribution of the positive matrix
\[
  \sigma^{-2k}X_n^kX_n^{*k}.
\]
Moment convergence to the \(k\)-th Fuss--Catalan law for powers of
non-Hermitian i.i.d.\ random matrices was proved in
\cite{alexeev-gotze-tikhomirov}; the almost sure empirical-distribution
convergence used here is the strengthened form in
\cite{alexeev-almost-sure-fuss}.  Hence \(\mu_{n,k}\) converges weakly almost
surely to the \(k\)-th Fuss--Catalan law \(\mu_k\).  The
moments of \(\mu_k\) are \(F_{k,m}\), and its right edge is \(\gamma_k^2\).
Alexeev states the theorem with \(X^T\) in the real case; the same
moment-and-variance proof applies to complex entries after replacing \(X^T\)
by \(X^*\).  In the trace expansion, an \(X^*\)-edge carries the conjugated
entry, the regular paths pair each entry with its conjugate and are counted in
the same way, while the non-regular paths are estimated as in the real case
after the usual truncation.

Fix \(\varepsilon>0\).  Since \(\gamma_k^2\) belongs to the support of \(\mu_k\),
\[
  \mu_k((\gamma_k^2-\varepsilon,\infty))>0.
\]
By the Portmanteau theorem, almost surely,
\[
  \liminf_{n\to\infty}\mu_{n,k}((\gamma_k^2-\varepsilon,\infty))
  \ge
  \mu_k((\gamma_k^2-\varepsilon,\infty))
  >
  0.
\]
Thus, for all sufficiently large \(n\), the matrix
\(\sigma^{-2k}X_n^kX_n^{*k}\) has an eigenvalue larger than
\(\gamma_k^2-\varepsilon\).  Equivalently,
\[
  \liminf_{n\to\infty}\sigma^{-2k}\|X_n^k\|^2
  \ge
  \gamma_k^2-\varepsilon.
\]
Letting \(\varepsilon\downarrow0\) gives the claimed lower bound.
\end{proof}

\section{The Proof of Corollary~\ref{cor:gaussian-powers}}
\label{sec:gaussian-proof}
\begin{proof}[Proof of Corollary~\ref{cor:gaussian-powers}]
Work in the tracial noncommutative probability space
\((\mathcal A,\tau)\) from Section~\ref{sec:preliminaries}.  Let
\[
  c=\frac{s_1+i s_2}{\sqrt2}
\]
be the circular element with \(\tau(cc^*)=1\).  Decompose \(Y_n\) as
\[
  H_{1,n}:=\frac{Y_n+Y_n^*}{\sqrt2},
  \qquad
  H_{2,n}:=\frac{Y_n-Y_n^*}{i\sqrt2}.
\]
Then \(H_{1,n}\) and \(H_{2,n}\) are independent normalized GUE matrices whose
off-diagonal entries have complex variance \(1/n\) and whose diagonal entries
are real centered Gaussians of variance \(1/n\).  With this diagonal convention,
each \(H_{r,n}\) converges strongly to the semicircular element \(s_r\), and
\[
  Y_n=\frac{H_{1,n}+iH_{2,n}}{\sqrt2}.
\]
Male's strong convergence theorem for independent Gaussian matrix families
\cite{male} gives, for every noncommutative polynomial \(P\),
\[
  \|P(Y_n,Y_n^*)\|\to \|P(c,c^*)\|
  \qquad \text{a.s.}
\]
Taking \(P(z,z^*)=z^k\), we obtain
\[
  \|Y_n^k\|\to \|c^k\|
  \qquad \text{a.s.}
\]

It remains to compute \(\|c^k\|\).  Since \(c^kc^{*k}\) has compactly supported
law,
\[
  \|c^k\|^2
  =
  \|c^kc^{*k}\|
  =
  \lim_{m\to\infty}
  \tau\bigl((c^kc^{*k})^m\bigr)^{1/m}.
\]
By the free Wick formula for the circular element, only non-crossing pairings
that match a copy of \(c\) with a copy of \(c^*\) contribute to
\(\tau((c^kc^{*k})^m)\).  For the word \((c^kc^{*k})^m\), these compatible
non-crossing pairings are exactly the pairings counted in
Lemma~\ref{lem:leading}.  Hence
\[
  \tau\bigl((c^kc^{*k})^m\bigr)
  =
  F_{k,m}
  =
  \frac1{km+1}\binom{(k+1)m}{m}.
\]
Therefore
\[
  \|c^k\|^2
  =
  \lim_{m\to\infty}
  \left[
  \frac1{km+1}\binom{(k+1)m}{m}
  \right]^{1/m}
  =
  \gamma_k^2,
\]
which proves \(\|Y_n^k\|\to\gamma_k\) almost surely.
\end{proof}

\section{The Proofs of Theorem~\ref{thm:sampled-products} and Corollary~\ref{cor:independent-products}}
\label{sec:cor-proof}

We first record a small consequence of the leading-word decomposition.

\begin{lemma}[Level preservation of leading pairings]
\label{lem:level-preservation}
In every balanced tree-like pairing of the boundary word
\[
  (z^kz^{*k})^\ell,
\]
a plus occurrence in the \(j\)-th level of a \(z^k\)-block is paired with a
minus occurrence in the corresponding \(j\)-th conjugate level of a
\(z^{*k}\)-block, where conjugate levels are indexed by the original factor
before the adjoint reversal.
\end{lemma}

\begin{proof}
Use the first-return decomposition in the proof of
Lemma~\ref{lem:leading}.  Put a root at the first occurrence of a root block.
The \(k\) plus occurrences of this block are closed by matching minus
occurrences in last-in-first-out order; otherwise the planar non-crossing
pairing would either cross a previous return or leave one of the \(k+1\)
subintervals with unmatched boundary.  Thus the first plus level is matched to
the first conjugate level, the second plus level to the second conjugate level,
and so on, where conjugate levels are indexed by the original factor before the
adjoint reversal.  The same argument applies recursively in each of the
\(k+1\) subintervals of the first-return decomposition.  Hence every leading
pair preserves the factor level.  For a fixed label word
\(\alpha=(\alpha_1,\ldots,\alpha_k)\), the two occurrences in such a pair
therefore carry the same matrix label \(\alpha_j\).  Consequently the leading
labeled tree-like words are counted by the same number \(F_{k,\ell}\),
uniformly in \(\alpha\), even when some labels repeat.
\end{proof}

\begin{proof}[Proof of Theorem~\ref{thm:sampled-products}]
We first prove the deterministic-index statement, simultaneously for all label
words in the fixed alphabet \(\{1,\ldots,N\}\).  We assume \(\sigma=1\); the
general case follows by scaling.  For
\(\alpha=(\alpha_1,\ldots,\alpha_k)\in\{1,\ldots,N\}^k\), put
\[
  Y_{\alpha,n}:=
  X_{\alpha_1,n}X_{\alpha_2,n}\cdots X_{\alpha_k,n}.
\]

Apply Lemma~\ref{lem:truncation} to each of the finitely many nested arrays
with labels \(1,\ldots,N\), and intersect the resulting full-probability
events.  After replacing the finitely many truncation parameters by their
maximum, we obtain truncated matrices \(\wt X_{a,n}\), \(1\le a\le N\), with a
common bound \(\eps_n\sqrt n\), where \(\eps_n(\log n)^B\to0\) for every fixed
\(B<\infty\), such that
\[
  \max_{1\le a\le N}\|X_{a,n}-\wt X_{a,n}\|\to0,
  \qquad
  \max_{1\le a\le N}\bigl(\|X_{a,n}\|+\|\wt X_{a,n}\|\bigr)=O(1)
  \qquad\text{a.s.}
\]
The second bound follows from the \(k=1\) case of the truncated moment
estimate and the same argument as in Theorem~\ref{thm:main}.  Hence, uniformly
in \(\alpha\in\{1,\ldots,N\}^k\),
\[
  \left\|
  Y_{\alpha,n}
  -
  \wt X_{\alpha_1,n}\wt X_{\alpha_2,n}\cdots \wt X_{\alpha_k,n}
  \right\|
  \le
  k M_n^{k-1}\max_{1\le a\le N}\|X_{a,n}-\wt X_{a,n}\|
  \to0
\]
almost surely, where
\(M_n=\max_{1\le a\le N}(\|X_{a,n}\|+\|\wt X_{a,n}\|)\).
It is therefore enough to prove the upper bound for the truncated products
\[
  \wt Y_{\alpha,n}:=
  \wt X_{\alpha_1,n}\wt X_{\alpha_2,n}\cdots \wt X_{\alpha_k,n}.
\]

For a fixed word \(\alpha\), expand
\[
  \E\Tr\bigl((\wt Y_{\alpha,n}\wt Y_{\alpha,n}^*)^m\bigr)
\]
as in the proof of Theorem~\ref{thm:main}.  The cyclic sign-and-label pattern is
\[
  \alpha_1,\ldots,\alpha_k,
  \alpha_k,\ldots,\alpha_1
\]
on each copy of \(Y_{\alpha,n}Y_{\alpha,n}^*\), with the second half conjugated.
We repeat the definitions of Section~\ref{sec:proof--main-proposition} with
the variable-edge type enlarged to include the matrix label.  Thus an
occurrence has type
\[
  (\lambda_s,u,v),
\]
where \(\lambda_s\) is the fixed matrix label assigned to that occurrence by
the boundary word and \((u,v)\) is the ordered coordinate pair.  Admissibility,
the number \(q_{\rm lab}\) of distinct variable-edges, and the defect
\[
  d_{\rm lab}:=(L+1-v)+(L-q_{\rm lab})
\]
are defined with respect to these labeled variable-edges.  Since the label
pattern is fixed by the boundary word, the labels impose only additional
restrictions on which occurrences may represent the same random variable.  The
pair-skeleton, corner-partition, scheme, and local-record encodings depend only
on the occurrence positions and coordinate identifications.  Consequently
Proposition~\ref{prop:defect-counting} applies without change and gives
\[
  \#\{\text{labeled admissible classes with defect }r\}
  \le F_{k,m}(Cm)^{Dr},
\]
uniformly over \(\alpha\in\{1,\ldots,N\}^k\).

For a word with labeled parameters
\[
  a=L-q_{\rm lab},\qquad b=L+1-v,
\]
the same moment estimate as in Appendix~\ref{app:moment} gives the contribution
bound
\[
  n\,\eps_n^{2a}n^{a-b}
  \#\{\text{labeled classes with defect }a+b\}.
\]
As before, admissibility gives \(b\ge a\).  Summing over \(a\) and
\(c=b-a\), and using the displayed labeled defect count, yields uniformly in
\(\alpha\)
\[
  \E\Tr\bigl((\wt Y_{\alpha,n}\wt Y_{\alpha,n}^*)^m\bigr)
  \le n(\gamma_k^2+\eta)^m
\]
for \(1\le m\le A\log n\).  Markov's inequality and the
Borel--Cantelli argument used for the upper bound in
Theorem~\ref{thm:main}, followed by a finite union bound over the \(N^k\)
possible words, give
\[
  \limsup_{n\to\infty}
  \max_{\alpha\in\{1,\ldots,N\}^k}\|\wt Y_{\alpha,n}\|
  \le \gamma_k
  \qquad\text{a.s.}
\]
The preceding truncation comparison transfers this estimate to
\(Y_{\alpha,n}\), uniformly in \(\alpha\).

It remains to prove the lower bound.  We first work with the truncated
products.  Let
\[
  \wt M_{\ell,n}(\alpha)
  :=
  \frac1n\Tr\bigl((\wt Y_{\alpha,n}\wt Y_{\alpha,n}^*)^\ell\bigr).
\]
For fixed \(\ell\ge1\), the moment-and-variance expansion is justified for
\(\wt M_{\ell,n}(\alpha)\), since the entries are bounded.  The expansion is
the fixed-\(\ell\) version of the labeled expansion above, with \(L=k\ell\).
Non-admissible labeled words vanish by centering.  Among admissible words, the
only terms of order one after the normalization \(n^{-L-1}\) are the balanced
tree-like words, namely those with \(q_{\rm lab}=L\) and \(v=L+1\).
If \(q_{\rm lab}<L\) or \(v<L+1\), then the fixed-\(\ell\) estimate gives
either a negative power of \(n\) or a positive power of the truncation parameter
\(\eps_n\).  Indeed, with
\[
  a=L-q_{\rm lab},\qquad b=L+1-v,
\]
the normalized contribution is bounded by \(\eps_n^{2a}n^{a-b}\), and
\(b\ge a\).  Thus if \((a,b)\ne(0,0)\), the contribution tends to zero either
because \(b>a\) or because \(a>0\).  This convergence is uniform over the finite
set \(\{1,\ldots,N\}^k\).
By Lemma~\ref{lem:level-preservation}, the leading labeled tree-like words are
counted by \(F_{k,\ell}\).  Therefore
\[
  \max_{\alpha\in\{1,\ldots,N\}^k}
  |\E\wt M_{\ell,n}(\alpha)-F_{k,\ell}|\to0 .
\]
The variance estimate is unchanged.  High-multiplicity labeled edges are
controlled by the same bounded-entry estimate as in Appendix~\ref{app:moment};
for fixed \(\ell\), the number of profiles is finite and the factors involving
\(\eps_n\) are harmless.  In the expansion of
\[
  \E\wt M_{\ell,n}(\alpha)^2,
\]
any connected gluing between the two traces
loses at least two free vertex labels, while disconnected leading pairs factor
into the square of the first moment.  The sampled labels impose only additional
independence constraints.  Equivalently, after grouping the expansion by the
finitely many combinatorial profiles, every non-factorizing profile contributes
\(O(n^{-2})+o(n^{-2})\) to the covariance, uniformly in \(\alpha\); the
\(o(n^{-2})\) terms are those carrying at least one extra truncation factor.
Hence, for every fixed \(\ell\),
\[
  \Var(\wt M_{\ell,n}(\alpha))\le C_{\ell,k,N}n^{-2},
\]
uniformly in \(\alpha\).  A finite union bound and Borel--Cantelli give
\[
  \max_{\alpha\in\{1,\ldots,N\}^k}
  |\wt M_{\ell,n}(\alpha)-F_{k,\ell}|
  \to0
  \qquad\text{a.s.}
\]

We now transfer the moment convergence to the original matrices.  Put
\[
  M_{\ell,n}(\alpha)
  :=
  \frac1n\Tr\bigl((Y_{\alpha,n}Y_{\alpha,n}^*)^\ell\bigr),
\]
and set
\[
  A_{\alpha,n}=Y_{\alpha,n}Y_{\alpha,n}^*,
  \qquad
  B_{\alpha,n}=\wt Y_{\alpha,n}\wt Y_{\alpha,n}^* .
\]
The upper bound already proved and the truncation comparison imply that
\[
  \max_{\alpha\in\{1,\ldots,N\}^k}
  \bigl(\|A_{\alpha,n}\|+\|B_{\alpha,n}\|\bigr)=O(1)
\]
almost surely.  Moreover, using
\[
  A_{\alpha,n}-B_{\alpha,n}
  =
  (Y_{\alpha,n}-\wt Y_{\alpha,n})Y_{\alpha,n}^*
  +
  \wt Y_{\alpha,n}(Y_{\alpha,n}-\wt Y_{\alpha,n})^*,
\]
we get \(\max_\alpha\|A_{\alpha,n}-B_{\alpha,n}\|\to0\) almost surely.  Thus
\[
  |M_{\ell,n}(\alpha)-\wt M_{\ell,n}(\alpha)|
  \le
  \|A_{\alpha,n}^\ell-B_{\alpha,n}^\ell\|
  \le
  \ell\max(\|A_{\alpha,n}\|,\|B_{\alpha,n}\|)^{\ell-1}
  \|A_{\alpha,n}-B_{\alpha,n}\|,
\]
and consequently
\[
  \max_{\alpha\in\{1,\ldots,N\}^k}
  |M_{\ell,n}(\alpha)-F_{k,\ell}|
  \to0
  \qquad\text{a.s.}
\]
for every fixed \(\ell\).
Taking a countable intersection over \(\ell\ge1\), we may assume that the preceding convergence holds for all \(\ell\) simultaneously.

Let \(\nu_{\alpha,n}\) be the empirical eigenvalue distribution of
\[
  Y_{\alpha,n}Y_{\alpha,n}^* .
\]
Since the upper bound gives a compact spectral window up to an \(o(1)\) error,
and since the Fuss--Catalan law is compactly supported and determined by its
moments, the preceding moment convergence implies
\[
  \nu_{\alpha,n}\Rightarrow\mu_k
  \qquad\text{a.s.}
\]
uniformly over \(\alpha\in\{1,\ldots,N\}^k\).  Since the label set
\(\{1,\ldots,N\}^k\) is finite, these weak convergences hold on a common
full-probability event.  The right edge of
\(\mu_k\) is \(\gamma_k^2\).  Therefore, for every \(\varepsilon>0\),
\[
  \mu_k((\gamma_k^2-\varepsilon,\infty))>0,
\]
and the Portmanteau theorem gives
\[
  \liminf_{n\to\infty}
  \min_{\alpha\in\{1,\ldots,N\}^k}
  \nu_{\alpha,n}\bigl((\gamma_k^2-\varepsilon,\infty)\bigr)
  \ge
  \mu_k((\gamma_k^2-\varepsilon,\infty))>0
  \qquad\text{a.s.}
\]
Hence, for every deterministic or randomly sampled label word,
\[
  \liminf_{n\to\infty}\|Y_{\alpha,n}\|
  \ge \gamma_k
  \qquad\text{a.s.}
\]
Letting \(\varepsilon\downarrow0\) proves the lower bound when \(\sigma=1\).
The general case follows by applying the normalized argument to
\(\sigma^{-1}W\) and multiplying the norm by \(\sigma^k\).

The preceding upper and lower bounds hold on one full-probability event
simultaneously for every deterministic label word with letters in
\(\{1,\ldots,N\}\).  Therefore they also hold for the randomly sampled label
word, which is independent of the matrices.
\end{proof}

\begin{proof}[Proof of Corollary~\ref{cor:independent-products}]
Apply Theorem~\ref{thm:sampled-products} with \(N=k\) and the deterministic
index word \(I_j=j\), \(1\le j\le k\).
\end{proof}

\section{Concluding Remarks}
\label{sec:conclusion}

We have established a fixed-product universality statement for non-Hermitian
random matrices with i.i.d.\ entries under the finite fourth moment assumption.
For every fixed \(k\), the repeated power \(X_n^k\), products sampled with
replacement from a fixed finite pool of independent copies, and the product of
\(k\) independent copies all have the same almost sure limiting norm
\[
  \sigma^k\gamma_k
  =
  \sigma^k\sqrt{\frac{(k+1)^{k+1}}{k^k}}.
\]
In this sense the freeness coefficient is a genuine limiting spectral norm
constant for these fixed noncommutative products, not only a formal constant
suggested by the leading moment calculation.

For comparison, in the symmetric Wigner analogue with the same variance
normalization, if \(S_n\) is one normalized symmetric matrix and
\(S_{1,n},\ldots,S_{k,n}\) are independent copies in a standard strong
asymptotic-freeness setting, then
\[
  \|S_n^k\|\to (2\sigma)^k,
  \qquad
  \|S_{1,n}\cdots S_{k,n}\|\to \sigma^k\gamma_k
  \qquad \text{a.s.}
\]
Thus the independent-product constant agrees with the non-Hermitian product
constant in the present paper, while the power case is different for \(k\ge2\).
The reason is structural: a symmetric matrix is normal, so \(S_n^k\) is governed
by the edge of the semicircle law through \(\|S_n^k\|=\|S_n\|^k\)
\cite{bai-yin-wigner}, whereas the non-Hermitian power \(X_n^k\) remains a
genuinely non-normal product and has the Fuss--Catalan singular-value edge.  For
independent products, by contrast, both models are governed by the same free
multiplicative-convolution edge
\cite{voiculescu-dykema-nica,nica-speicher,haagerup-thorbjornsen,male}.

Our main technical challenge is a direct high-moment argument.  After truncation,
the leading trace words are the balanced tree-like words counted by the
Fuss--Catalan number \(F_{k,m}\).  The essential point is to control the
admissible non-leading words created by matrix products, especially the extra
collision patterns that occur when the same matrix is reused in every factor of
\(X_n^k\).  This is done through the defect-sensitive enumeration in
Proposition~\ref{prop:defect-counting}, where the loss over the Fuss--Catalan
leading count is only polynomial in the logarithmic moment parameter and
therefore remains summable in the moment method.

The comparison with the standard Bai--Yin power estimate should be read in this
framework.  In the power case, the product-limit theorem yields the sharper
upper edge \(\sigma^k\gamma_k\), which has the correct \(\sqrt{k}\)-scale for
fixed \(k\), rather than the coarser \(k+1\) scale.  This improvement is
therefore an application of the fixed-product limit, while the main phenomenon
is the stability of the freeness coefficient across repeated, fixed-pool
sampled, and independent products.

\appendix

\section{Auxiliary Proofs for the Reduction and Truncation}
\label{app:reduction}

The auxiliary estimates are stated in the form needed for the nested almost
sure truncation.  They are consequences of the Bai--Yin upper-bound
moment method \cite{bai-yin}; the classical truncation and centralization
scheme appears, for instance, in \cite[Lemmas~2.2--2.3]{yin-bai-krishnaiah}.

\begin{lemma}[Bounded fourth-moment Bai--Yin estimate]
\label{lem:bounded-fourth-moment-by}
Fix \(A>0\) and \(K,K_4<\infty\).  Let
\(Z_n=(z_{ij}^{(n)})_{1\le i,j\le n}\)
have independent centered entries, not necessarily identically distributed,
such that, for some \(0<\delta\le1\),
\[
  \E|z_{ij}^{(n)}|^2\le\delta^2,
  \qquad
  \E|z_{ij}^{(n)}|^4\le K_4\delta^4,
  \qquad
  |z_{ij}^{(n)}|\le K\delta\sqrt n
  \qquad(1\le i,j\le n).
\]
Then there is a constant \(C=C(A,K,K_4)\) such that, uniformly for
\(1\le m\le A\log n\),
\[
  \E\Tr\left(\left(n^{-1}Z_nZ_n^*\right)^m\right)
  \le
  n(C\delta)^{2m}
  \qquad\text{for all sufficiently large }n.
  \tag{A.1}
\]
\end{lemma}

\begin{proof}
This is the bounded-entry, bounded-fourth-moment form of the upper-bound
moment enumeration in Bai--Yin \cite{bai-yin}; see also the proof of the
Bai--Yin power estimate and the truncation discussion in
\cite[Chapter~5]{bai-silverstein}.  The proof
of the quoted estimate uses only the following uniform inputs: centering,
independence, the bounds on the second and fourth moments, and the entrywise
bound.  Hence the same enumeration applies verbatim to the present triangular,
non-identically distributed array.  We use only this non-sharp form for the
truncation reduction.

We record how the quoted estimate matches the present triangular form.  In
the usual trace expansion of
\(\E\Tr((n^{-1}Z_nZ_n^*)^m)\), independence and centering remove every word in
which some matrix entry occurs only once.  For a contributing word, the
hypotheses give the uniform multiplicity bounds
\[
  \E |z_{ij}^{(n)}|^2\le \delta^2,\qquad
  \E |z_{ij}^{(n)}|^4\le K_4\delta^4,\qquad
  |z_{ij}^{(n)}|\le K\delta\sqrt n .
\]
These are exactly the inputs used in the bounded-entry Bai--Yin upper-bound
enumeration.  The enumeration is non-sharp: the tree-like exploration gives the
usual Catalan exponential factor, while the vertex-deficit and high-multiplicity
records are paid for by the powers of \(n^{-1}\) and by the fourth-moment
control.  More concretely, in the usual profile summation one first keeps the
edge multiplicities: edges of multiplicity two use the second-moment bound,
edges of multiplicity three are controlled by interpolation,
\[
  \E|z_{ij}^{(n)}|^3
  \le
  \bigl(\E|z_{ij}^{(n)}|^2\bigr)^{1/2}
  \bigl(\E|z_{ij}^{(n)}|^4\bigr)^{1/2}
  \le K_4^{1/2}\delta^3,
\]
and edges of multiplicity at least four use the fourth-moment bound before the
entrywise bound is applied.  The vertex deficit supplies the remaining powers
of \(n^{-1}\).  Only after these compensations are made does one sum the
profiles, whose number is bounded by \(C(A)^m\) for \(m\le A\log n\).  Thus the
upper-bound enumeration is uniform for \(1\le m\le A\log n\), and the remaining
polynomial factors in the word profiles are absorbed into a constant
\(C=C(A,K,K_4)\), independent of the particular triangular array.  The argument
uses only the uniform bounds above, not identical distribution of the entries.
This gives \((\mathrm{A}.1)\).
\end{proof}

\begin{lemma}[Nested small-variance Bai--Yin principle]
\label{fact:triangular-bai-yin}
Let \((x_{ij})_{i,j\ge1}\) be a single infinite array of i.i.d. copies of a
real- or complex-valued mean-zero random variable \(x\) with
\(\E|x|^4<\infty\).  Let \(f_n\) be measurable functions such that
\(|f_n(u)|\le |u|\).  Define
\[
  z_{ij}^{(n)}
  =
  f_n(x_{ij})-\E f_n(x),
  \qquad
  Z_n=(z_{ij}^{(n)})_{1\le i,j\le n}.
\]
If
\begin{equation} \label{eqn:conditions}
  \E |f_n(x)|^4\to0,
\end{equation}
then
\[
  \|n^{-1/2}Z_n\|\to0
  \qquad\text{a.s.}
\]
\end{lemma}

\begin{proof}[Proof of Lemma~\ref{fact:triangular-bai-yin}]
Fix \(0<\delta\le1\), and split
\[
  f_n(x)
  =
  f_n(x)\one_{\{|x|\le \delta\sqrt n\}}
  +
  f_n(x)\one_{\{|x|>\delta\sqrt n\}} .
\]
The second part is eventually absent from the \(n\times n\) upper-left corner.
Indeed, by Borel--Cantelli on the shells \(\max(i,j)=r\),
\[
  \sum_{r\ge1} r\,\Prob\{|x|>\delta\sqrt r\}<\infty
\]
because \(\E|x|^4<\infty\): for example,
\[
  \sum_{r\ge1} r\,\one_{\{|x|>\delta\sqrt r\}}
  \le
  C_\delta |x|^4
\]
pointwise after integrating over \(x\).  Hence, almost surely, only finitely
many entries overall satisfy
\(|x_{ij}|>\delta\sqrt{\max(i,j)}\).  For an entry in the shell
\(r=\max(i,j)\le n\), we have
\(|x_{ij}|>\delta\sqrt n\Rightarrow |x_{ij}|>\delta\sqrt r\).  The finitely
many exceptional shell entries are
fixed finite random variables, so for all large \(n\) they also satisfy
\(|x_{ij}|\le\delta\sqrt n\).  This gives the \(n\)-dependent threshold and
hence
\[
  \limsup_{n\to\infty}
  \max_{1\le i,j\le n}|x_{ij}|/\sqrt n\le\delta
  \qquad\text{a.s.}
\]
The deterministic centering coming from the discarded
\(\{|x|>\delta\sqrt n\}\) part has rank one and norm
\[
  \sqrt n\,\bigl|\E[f_n(x)\one_{\{|x|>\delta\sqrt n\}}]\bigr|
  \le
  \frac{\E[|x|^4\one_{\{|x|>\delta\sqrt n\}}]}{\delta^3 n}
  =o(1).
\]
It is therefore enough to estimate the centered bounded array
\[
  z_{ij}^{(n,\delta)}
  =
  f_n(x_{ij})\one_{\{|x_{ij}|\le\delta\sqrt n\}}
  -
  \E[f_n(x)\one_{\{|x|\le\delta\sqrt n\}}].
\]
Its entries are independent, centered, bounded by \(2\delta\sqrt n\) for all
large \(n\), and have variance
\[
  \theta_{n,\delta}^2:=\E|z_{11}^{(n,\delta)}|^2
  \le
  4\bigl(\E|f_n(x)|^4\bigr)^{1/2}\to0 .
\]
For the same reason,
\[
  \E|z_{11}^{(n,\delta)}|^4
  \le
  C\,\E|f_n(x)|^4\to0,
\]
so, for this fixed \(\delta\), both
\(\E|z_{11}^{(n,\delta)}|^2\le\delta^2\) and
\(\E|z_{11}^{(n,\delta)}|^4\le\delta^4\) hold for all sufficiently large
\(n\).

If \(Z_n^{(\delta)}=(z_{ij}^{(n,\delta)})_{1\le i,j\le n}\), then
Lemma~\ref{lem:bounded-fourth-moment-by}, applied with \(K=2\), \(K_4=1\),
and \(\delta\) fixed, gives for every fixed \(A>0\), uniformly for
\(1\le m\le A\log n\),
\[
  \E\Tr\left(\left(n^{-1}Z_n^{(\delta)}
  (Z_n^{(\delta)})^*\right)^m\right)
  \le
  n(C\delta)^{2m}
  \qquad\text{for all sufficiently large }n,
\]
where \(C=C(A)\).  This is the only place where the non-sharp Bai--Yin
bounded-entry enumeration is used.

Choose \(m=\lfloor A\log n\rfloor\) with \(A\) large.  From this estimate and Markov's
inequality,
\[
  \sum_n
  \Prob\{\|n^{-1/2}Z_n^{(\delta)}\|>2C\delta\}<\infty .
\]
Hence
\[
  \limsup_{n\to\infty}\|n^{-1/2}Z_n^{(\delta)}\|\le 2C\delta
  \qquad\text{a.s.}
\]
The same full-probability event can be chosen for all rational
\(\delta>0\).  To pass back to the original centered array, write
\[
  z_{ij}^{(n)}
  =
  z_{ij}^{(n,\delta)}
  +
  f_n(x_{ij})\one_{\{|x_{ij}|>\delta\sqrt n\}}
  -
  \E\bigl[f_n(x)\one_{\{|x|>\delta\sqrt n\}}\bigr].
\]
The random matrix with entries
\(f_n(x_{ij})\one_{\{|x_{ij}|>\delta\sqrt n\}}\) is eventually zero in the
\(n\times n\) corner, and its centering term is rank one with norm \(o(1)\), as
shown above.  Letting \(\delta\downarrow0\) gives the desired convergence to
zero.
\end{proof}

\begin{lemma}[Small-variance tail estimate]
\label{lem:small-variance-bai-yin}
Let \(x\) be a real- or complex-valued mean-zero random variable with
\(\E |x|^4<\infty\).  Let
\(a_n\to\infty\), and let
\((x_{ij})_{i,j\ge1}\) be a single infinite array of i.i.d. copies of \(x\).
Define the centered tail array \(Y_n= (y_{ij}^{(n)})_{1\le i,j\le n}\) by
\[
  y_{ij}^{(n)}
  =
  x_{ij}\one_{\{|x_{ij}|>a_n\}}
  -\E\bigl[x\one_{\{|x|>a_n\}}\bigr].
\]
If
\begin{equation} \label{eqn:conditions2}
  \E\bigl[|x|^4\one_{\{|x|>a_n\}}\bigr]\to0,
\end{equation}
then
\[
  \|n^{-1/2}Y_n\|=o(1)
  \qquad \text{a.s.}
\]
\end{lemma}

\begin{proof}[Proof of Lemma~\ref{lem:small-variance-bai-yin}]
Apply Lemma~\ref{fact:triangular-bai-yin} with
\[
  f_n(u)=u\one_{\{|u|>a_n\}}.
\]
Its hypothesis holds because
\[
  \E |f_n(x)|^4
  =
  \E\bigl[|x|^4\one_{\{|x|>a_n\}}\bigr]\to0.
\]
For \(y_{11}^{(n)}=x\one_{\{|x|>a_n\}}-\E[x\one_{\{|x|>a_n\}}]\), put
\[
  \tau_n^2:=\E|y_{11}^{(n)}|^2 .
\]
This variance parameter also tends to zero:
\[
  \tau_n^2
  \le
  \E\bigl[|x|^2\one_{\{|x|>a_n\}}\bigr]
  +
  \left|\E[x\one_{\{|x|>a_n\}}]\right|^2
  \to0 .
\]
This is the claimed convergence.
\end{proof}

\begin{proof}[Proof of Lemma~\ref{lem:truncation}]
Choose, for example, \(a_n=n^{1/4}\), and write
\(\eps_n^{(0)}=a_n/\sqrt n=n^{-1/4}\).  Then \(a_n\to\infty\),
\(\eps_n^{(0)}(\log n)^B\to0\) for every fixed \(B\), and finite fourth moment
gives
\[
  \E\bigl[|w_{11}|^4\one_{\{|w_{11}|>a_n\}}\bigr]\to0,
  \qquad
  \E\bigl[|w_{11}|^2\one_{\{|w_{11}|>a_n\}}\bigr]\to0 .
\]
Set
\[
  \widehat w_{ij}
  :=
  w_{ij}\one_{\{|w_{ij}|\le a_n\}}
  -\E\bigl[w_{11}\one_{\{|w_{11}|\le a_n\}}\bigr],
\]
and let \(\widehat\sigma_n^2=\E|\widehat w_{11}|^2\).  Then
\(\widehat\sigma_n\to1\), and we put
\(\wt w_{ij}=\widehat\sigma_n^{-1}\widehat w_{ij}\).  After replacing
\(\eps_n^{(0)}\) by \(\eps_n=3\eps_n^{(0)}\), which still satisfies
\(\eps_n(\log n)^B\to0\) for every fixed \(B\), we have
\(|\wt w_{ij}|\le\eps_n\sqrt n\) for all large \(n\).  Indeed,
\(|\E[w_{11}\one_{\{|w_{11}|>a_n\}}]|\to0\), so the deterministic centering is
eventually smaller than \(a_n\).  Here hats denote the centered truncation and
tildes denote the subsequent variance normalization.

The discarded tail is controlled as follows.  Since \(\E w_{11}=0\),
\[
  \E[w_{11}\one_{\{|w_{11}|\le a_n\}}]
  =
  -\E[w_{11}\one_{\{|w_{11}|>a_n\}}],
\]
and hence \(W_n-\widehat W_n\) is exactly the centered tail matrix:
\[
  W_n-\widehat W_n
  =
  \left(
  w_{ij}\one_{\{|w_{ij}|>a_n\}}
  -\E w_{11}\one_{\{|w_{11}|>a_n\}}
  \right)_{ij}.
\]
Denote the variance of these centered tail entries by \(\tau_n^2\).  Then
\[
  \tau_n^2
  \le
  \E\bigl[|w_{11}|^2\one_{\{|w_{11}|>a_n\}}\bigr]+o(1)\to0.
\]
The centered tail entries also satisfy the Bai--Yin fourth-moment truncation
condition because
\[
  \E\left|
    w_{11}\one_{\{|w_{11}|>a_n\}}
    -\E w_{11}\one_{\{|w_{11}|>a_n\}}
  \right|^4
  \le
  C\,\E[|w_{11}|^4\one_{\{|w_{11}|>a_n\}}]\to0 .
\]
Hence
Lemma~\ref{lem:small-variance-bai-yin} gives
\[
  \left\|n^{-1/2}(W_n-\widehat W_n)\right\|\to0
  \qquad\text{a.s.}
\]
The centered truncated entries have \(\widehat\sigma_n^2\to1\), uniformly
bounded fourth moments, and satisfy
\(|\widehat w_{ij}|\le 2\sqrt n\) for all large \(n\).  A fixed rescaling puts
them under Lemma~\ref{lem:bounded-fourth-moment-by} with
\(\delta=1\).  The same Markov--Borel--Cantelli argument as above therefore
gives \(\|n^{-1/2}\widehat W_n\|=O(1)\) almost surely.  Consequently
\[
  \|n^{-1/2}(\wt W_n-\widehat W_n)\|
  =
  |\widehat\sigma_n^{-1}-1|\,\|n^{-1/2}\widehat W_n\|
  =o(1).
\]
Combining the last two displays proves the lemma.
\end{proof}

\section{Proofs for the Word Enumeration and Defect Counting}
\label{app:combinatorics}

\begin{proof}[Proof of Lemma~\ref{lem:leading}]
If \(\mathbf i\) is balanced tree-like, then \(q=L\) and \(v=L+1\), so the
quotient graph traced by the word is a tree with \(L\) edges.  The closed
contour of a tree crosses each edge exactly twice, once in each direction;
equivalently, every ordered variable-edge is represented by one plus
occurrence and one minus occurrence.  Since a tree embedded with one boundary
contour is planar, the two visits to each edge give a non-crossing matching of
the \(L\) plus occurrences with the \(L\) minus occurrences in
\[
  (z^kz^{*k})^m .
\]
Conversely, a compatible planar non-crossing matching of plus and minus
occurrences glues the contour into a rooted plane tree, and reading its
contour recovers a unique balanced tree-like word up to relabeling of
vertices.  Thus balanced tree-like words are precisely these planar
non-crossing pairings.

Put a root at the first occurrence and follow the contour.  The first-return
decomposition of a compatible non-crossing matching gives the recursion:
the first block \(z^kz^{*k}\) has \(k\) plus occurrences, and planarity forces
their matching minus occurrences to close in last-in-first-out order.  The
intervals of the boundary lying before the first return, between consecutive
returns, and after the last return are therefore \(k+1\) ordered subwords,
each again carrying a compatible non-crossing matching.  Each piece is read
with the natural root and orientation inherited from its matched boundary arc;
with this convention it is again of the same block-compatible type and is
counted by the same generating function \(M(t)\).  Thus no variable-length
partial-block datum is introduced in the leading decomposition.  This gives a
formal decomposition
\[
  \text{root block}\quad\longleftrightarrow\quad
  (T_0,T_1,\ldots,T_k),
\]
where \(T_j\) is the rooted object inside the \(j\)-th region cut out by the
last-in-first-out returns, and the numbers of complete \(z^kz^{*k}\)-blocks in
the \(T_j\)'s sum to \(m-1\).  Hence the pieces are block-compatible objects
of the same class, possibly empty, not new partial-block species.  This gives a
bijection between balanced tree-like words and rooted ordered
\((k+1)\)-ary trees with \(m\) internal vertices.  Let \(M_m\) denote their
number and set
\[
  M(t)=1+\sum_{m\ge1}M_m t^m .
\]
Under this bijection, deleting the root internal vertex leaves \(k+1\) ordered
tree-like subwords, possibly empty, whose numbers of internal vertices sum to
\(m-1\).  Conversely, any such ordered \((k+1)\)-tuple is inserted into the
\(k+1\) slots of the root and reconstructs a unique word.  Therefore
\[
  M(t)=1+tM(t)^{k+1}.
\]
Lagrange inversion yields
\[
  [t^m]M(t)
  =
  \frac1{km+1}\binom{(k+1)m}{m}.
\]
This proves the claim.
\end{proof}

The entropy bound
\[
  \binom{N}{pN}
  \le p^{-pN}(1-p)^{-(1-p)N}
\]
with \(N=(k+1)m\) and \(p=1/(k+1)\) gives
\(\binom{(k+1)m}{m}\le \gamma_k^{2m}\), and hence \eqref{eq:F-growth}.

We also record the promised verification of the zero-defect case.  If
\(d(\mathbf i)=0\), then \(q=L\) and \(v=L+1\), so the connected graph traced
by the word is a tree and every ordered variable-edge occurs exactly twice.  A
closed contour on a tree crosses each cut equally often in the two directions;
hence the two occurrences of every edge have opposite signs.  This is exactly
the balanced tree-like condition.

\begin{definition}[Decorated leading envelopes]
A decorated leading envelope is a counting object built from a leading
tree-like contour counted by the generating function \(M(t)\).  It consists of
that canonical leading contour together with:
\begin{enumerate}[label=\textup{(\roman*)},leftmargin=2em]
\item one or two distinguished boundary corners, depending on whether the
  envelope is rooted or two-rooted;
\item finitely many additional marked boundary corners;
\item for each retained mark, a finite boundary type recording its sign and its
  offset inside a \(z^k\) or \(z^{*k}\) block;
\item the numbers of complete \(z^kz^{*k}\)-blocks in the intervals between
  consecutive retained marks.
\end{enumerate}
The signs in the envelope are the canonical block-compatible leading signs.
They are not the signs of the original scheme-edge expansion.  Thus an envelope
is used only to count a plane tree shape with marked boundary data; the original
same-sign or opposite-sign information is recovered later after the envelope is
placed back into the fixed word \((z^kz^{*k})^m\).
\end{definition}

\begin{definition}[Fixed scheme records]
In the sequel a fixed scheme record means the following finite data: the rooted
abstract one-face multigraph obtained after pruning and suppressing degree-two
paths; the pairing of its half-edges into scheme edges; the cyclic order at
each scheme vertex, equivalently the rooted boundary tour; the finite
endpoint/side convention attached to each scheme edge; and the \(O(p)\)
attachment positions, together with their finite cut offsets, in the fixed word
\((z^kz^{*k})^m\).  Loops and multiple edges are allowed.  This record does not
contain the occurrence labels along a suppressed degree-two chain, nor a
same-sign/opposite-sign sequence along that chain.
\end{definition}

The next lemmas give an injective encoding.  In every use below the retained
data are only: the finite
scheme, the \(O(p)\) attachment positions and cut offsets, the finite endpoint
convention for each scheme edge, and the leading envelopes with their marked
corners.  No list of labels along a long degree-two chain is retained.  Once the
marked arcs are placed back into the fixed cyclic word, the missing labels are
read in cyclic order, and the same-sign or opposite-sign type of each restored
pair is then determined by those labels.  Thus the polynomial factor
\((Cm)^{O(p)}\) comes only from scheme data and marked positions; there is no
hidden factor exponential in the lengths of the suppressed chains.

\begin{lemma}[Edge-expansion normal form]
\label{lem:edge-expansion-normal-form}
Fix the abstract rooted scheme obtained from a pair skeleton after pruning
trees and suppressing degree-two paths.  For each scheme edge, its preimage in
the original skeleton is encoded by the following data:
\begin{enumerate}[label=\textup{(\roman*)},leftmargin=2em]
\item two root half-edges, one at each endpoint of the scheme edge;
\item two boundary strings \(I=(i_1,\ldots,i_\ell)\) and
  \(J=(j_1,\ldots,j_\ell)\) of occurrence positions along the face tour;
\item rooted plane branches attached in the gaps along these two strings;
\item a finite endpoint convention \(\theta\) saying whether \(i_r\) is paired
  with \(j_r\) or with \(j_{\ell+1-r}\), together with the two possible
  choices of which side is read first.
\end{enumerate}
The path pairs of the expanded scheme edge are then exactly
\[
  \{i_r,j_{\theta(r)}\},\qquad 1\le r\le \ell,
\]
with \(\ell\ge1\).  Empty gaps between consecutive retained marks are allowed;
they simply carry zero complete \(z^kz^{*k}\)-blocks.  The attached branches
are the tree components removed during pruning.
The same normal form applies to loop scheme edges: the two root half-edges may
be incident to the same scheme vertex, but cutting at them still separates the
loop expansion into the two boundary strings \(I\) and \(J\).  For a loop, the
first side is declared by the rooted boundary tour, and the remaining ambiguity
is one of the finite endpoint conventions included in \(\theta\).  Thus
\(\theta\) ranges over a fixed finite set of side-choice, orientation, and
loop annulus/M\"obius conventions; its size is independent of the length
\(\ell\).  Conversely, these data reconstruct the expansion of that scheme
edge.
\end{lemma}

\begin{proof}
After all degree-one branches have been removed, every internal vertex on the
preimage of a scheme edge has degree two in the remaining two-core.  Therefore
the non-branch part between the two scheme endpoints is a single path.  The
one-face boundary runs along the two sides of this path, producing two ordered
strings of side labels.  At the unsigned path level, pairing the two sides has
only the forward/reverse and side-choice alternatives listed in \(\theta\).
This is one global convention for the whole degree-two chain, not an
independent choice at each internal edge.
Equivalently, the regular neighbourhood of a degree-two chain whose core is an
interval has two boundary arcs; once the endpoint sides are fixed, the opposite
side of the \(r\)-th edge is determined up to one global reversal.  In the
single-cycle exception the neighbourhood is an annulus or a M\"obius band, and
this global ambiguity is one of the same finite endpoint/side conventions.
The signed same/opposite status of each path pair is not recorded here; it is
recovered later from the fixed occurrence labels in
Lemma~\ref{lem:edge-expansion-injection}.  All other components incident to
vertices of this path are rooted trees, because any
cycle in such a component would have survived in the two-core and would belong
to the scheme.  This proves the normal form, and the converse gluing is
immediate from the same data.
\end{proof}

\begin{lemma}[Pruned components are leading branches]
\label{lem:pruned-leading-branch}
Every tree component removed during the pruning step is, after rooting at its
attachment corner, a leading tree-like branch, up to the finite boundary offset
marks produced by cuts inside \(z^k\) or \(z^{*k}\).
\end{lemma}

\begin{proof}
Such a component is attached to the remaining two-core by a bridge.  Removing
any edge inside the component also separates the quotient graph into two
pieces.  Since the original object has one boundary tour, that tour crosses
the corresponding cut exactly twice, once in each direction.  Here ``opposite
directions'' refers to the traversal direction of the underlying graph edge,
not to an orientability convention for the glued surface.  With the
definition of ordered variable-edges used in Section~3.1, the two occurrences of
the bridge therefore have opposite signs.  Applying this argument to every
edge of the tree component shows that all its pairs are plus/minus balanced.
The rooted contour of a tree is planar, so the component is a rooted leading
tree-like branch.  If the attachment cut lies inside a \(z^k\) or \(z^{*k}\)
block, the offset is one of the finite boundary marks already recorded in the
scheme expansion.
\end{proof}

\begin{lemma}[Injective normalization of a scheme-edge expansion]
\label{lem:edge-expansion-injection}
For a fixed scheme edge, fixed endpoint occurrence positions, fixed finite
cut-offset data, and fixed endpoint convention \(\theta\), the expansions in
Lemma~\ref{lem:edge-expansion-normal-form} inject into decorated two-rooted
leading envelopes with a bounded number of marked corners, together with the
same endpoint convention.
\end{lemma}

\begin{proof}
We define the normalization map explicitly.  Given an expanded scheme edge
\({\mathcal E}\), first cut it at its two root half-edges.  This produces two
boundary strings \(I=(i_1,\ldots,i_\ell)\) and
\(J=(j_1,\ldots,j_\ell)\), rooted branches attached in the gaps, and the
endpoint convention \(\theta\) from
Lemma~\ref{lem:edge-expansion-normal-form}.  The map
\[
  {\mathcal N}_\theta:{\mathcal E}\longmapsto {\mathcal T}
\]
keeps the following data:
\begin{enumerate}[label=\textup{(\alph*)},leftmargin=2em]
\item the plane rooted shape obtained from the two boundary strings and all
  attached branches;
\item the distinguished path with its two root corners; only the finitely many
  endpoint and cut offsets are recorded as additional marked corners;
\item the numbers of complete \(z^kz^{*k}\)-blocks placed in each interval
  between consecutive retained boundary marks.
\end{enumerate}
Equivalently, the record attached to this scheme edge is
\[
  {\mathcal R}_{\mathcal E}
  =
  \bigl({\mathcal T},\theta,\text{ endpoint cuts},
  \text{ cut offsets},\text{ interval block counts}\bigr).
\]
We denote this record by
\[
  \operatorname{Enc}_\theta({\mathcal E})={\mathcal R}_{\mathcal E}.
\]
It forgets only whether a path pair is same-sign or opposite-sign.  After this
forgetting, the distinguished path is interpreted with the canonical
opposite-sign convention, and the same plane tree shape is given the canonical
block-compatible leading boundary labelling determined by the recorded complete
block counts and boundary offsets.  Thus \({\mathcal T}\) is a decorated
two-rooted leading envelope.  This is only a counting envelope; it is not
asserted that the original scheme-edge expansion has become a leading
sub-skeleton.  The original signed chain is recovered only after the envelope is
placed back into the fixed word \((z^kz^{*k})^m\) with the recorded endpoint
offsets and block counts.

The inverse decoder on the image is as follows.  Given a record
\({\mathcal R}_{\mathcal E}\), define
\(\operatorname{Dec}_\theta({\mathcal R}_{\mathcal E})\) as follows.  First
place the two endpoint cuts and the finitely many cut offsets back in the fixed
word \((z^kz^{*k})^m\).  These are \(O(1)\) labels for this scheme edge, not
the full list of labels along the degree-two chain.  Then read every unmarked
interval in the original cyclic order.  The template \({\mathcal T}\) tells how
many complete blocks and which nested rooted branches occupy each interval, so
the labels in \(I\) and \(J\) are restored in order.  Finally pair the restored
entries by
\[
  i_r\longleftrightarrow j_{\theta(r)},\qquad 1\le r\le \ell,
\]
and reattach the rooted branches in their recorded attachment gaps.  The sign
type of this pair is then read from the fixed signs of the two restored
occurrence labels:
opposite signs give an untwisted edge, while equal signs give a twisted edge.
Equivalently, the reconstruction checklist is
\[
  ({\mathcal T},\theta,\text{ endpoint offsets},\text{ marked block counts})
  \Longrightarrow (I,J)
  \Longrightarrow
  \{\,i_r\leftrightarrow j_{\theta(r)}:1\le r\le\ell\,\}.
\]
In particular, no \(2^\ell\)-type choice is introduced along a degree-two
chain.  The same-sign or opposite-sign type of each restored pair is determined
by the fixed signs of the restored occurrence positions.  Thus no
variable-length twist sequence is stored as extra data; it is recovered edge by
edge from the restored labels.  The construction gives
\[
  \operatorname{Dec}_\theta(\operatorname{Enc}_\theta({\mathcal E}))
  =
  {\mathcal E},
\]
which proves injectivity.
\end{proof}

\begin{remark}[Block-compatible domination]
The envelope used in Lemma~\ref{lem:edge-expansion-injection} is not an
arbitrary plane-tree encoding.  Once the \(O(p)\) scheme attachment cuts and
their offsets are fixed, every unmarked interval is read in the original cyclic
order of the word \((z^kz^{*k})^m\).  Thus the only free data in an edge
envelope are the leading branch shapes, the two distinguished roots, the
finitely many endpoint/cut offsets, and the block counts of the resulting
intervals.  Consequently, when all scheme edges and vertex sectors are listed
in the boundary order, the resulting objects form an ordered forest of
block-compatible leading envelopes with \(O(p)\) roots and \(O(p)\) marked
corners.  This is the class counted by the coefficients of \(M(t)^r\) in
Lemma~\ref{lem:coefficient-comparison}; no Catalan count of unrestricted plane
trees is being used.
\end{remark}

\begin{lemma}[No twist-sequence record]
\label{lem:no-twist-record}
Fix a scheme edge, its two endpoint cuts, the finite endpoint convention
\(\theta\), and the decorated two-rooted envelope produced by
Lemma~\ref{lem:edge-expansion-injection}.  After this envelope is placed back
into the fixed cyclic word \((z^kz^{*k})^m\), the same-sign or opposite-sign
type of every pair along the restored degree-two chain is uniquely determined.
In particular, a chain of length \(\ell\) contributes no factor \(2^\ell\).
\end{lemma}

\begin{proof}
By the reconstruction in Lemma~\ref{lem:edge-expansion-injection}, the
endpoint cuts, block counts, and boundary offsets restore the two side strings
\[
  I=(i_1,\ldots,i_\ell),\qquad J=(j_1,\ldots,j_\ell)
\]
as occurrence labels in the original cyclic word.  The endpoint convention
\(\theta\) then gives the pairs \(i_r\leftrightarrow j_{\theta(r)}\).  The
sign attached to each occurrence label is fixed once and for all by the word
\((z^kz^{*k})^m\).  Therefore each restored pair is automatically classified
as opposite-sign or same-sign.  This classification is read from the labels; it
is not an additional sequence of binary choices.
\end{proof}

\begin{example}[Schematic scheme-edge expansion]
For one expanded scheme edge, the two sides of its degree-two chain may be
read as
\[
  I=(i_1,\ldots,i_\ell),\qquad
  J=(j_1,\ldots,j_\ell).
\]
If the endpoint convention is the reversed one, the path pairs are
\[
  i_1\leftrightarrow j_\ell,\quad
  i_2\leftrightarrow j_{\ell-1},\quad \ldots,\quad
  i_\ell\leftrightarrow j_1 .
\]
The envelope records the two side strings, the endpoint convention, the root
corners, the finite cut offsets, and the leading-tree branches in the gaps.  It
does not record a separate sign choice for each displayed pair.  Once the
labels \(i_r,j_{\ell+1-r}\) are placed back in the fixed word
\((z^kz^{*k})^m\), their signs determine whether that pair is untwisted or
twisted.
\end{example}

\begin{lemma}[Coefficient comparison for marked leading forests]
\label{lem:coefficient-comparison}
Let \(M(t)=1+tM(t)^{k+1}\), and fix \(C_*<\infty\).  There is a constant
\(C_0=C_0(k,C_*)\) such that, whenever \(1\le p\le C_*m\),
\[
  0\le r,j\le C_*p,\qquad |h|\le C_*p,
\]
one has
\[
  \bigl(C_*(m+|h|+r+1)\bigr)^j [t^{m+h}]M(t)^r
  \le
  F_{k,m}(C_0m)^{C_0p},
\]
with the convention that a coefficient with negative index is zero.
\end{lemma}

\begin{proof}
Put \(N=m+h\).  If \(N<0\), there is nothing to prove.  Otherwise
\(N\le Cm\), because \(p\le C_*m\).  The factor
\((C_*(m+|h|+r+1))^j\) is at most \((Cm)^{C_*p}\).  It is written this way
rather than as \((t\,d/dt)^j\) because ordered boundary marks may be placed on
root-only or empty components; the harmless \(+1\) covers the case \(N=r=0\).
If \(N=0\), the coefficient of \(M(t)^r\) is \(1\), and if \(r=0\), then
\(M(t)^0=1\) so the coefficient is nonzero only for \(N=0\).  In both cases
the displayed factor is absorbed into the right-hand side after increasing
\(C_0\).  For \(N\ge1\) and \(r\ge1\), Lagrange inversion gives
\[
  [t^N]M(t)^r
  =
  \frac{r}{(k+1)N+r}\binom{(k+1)N+r}{N}.
\]
Write
\[
  A(N,r):=[t^N]M(t)^r
  =
  \frac{r}{(k+1)N+r}\binom{(k+1)N+r}{N}
  \qquad (r\ge1).
\]
The adjacent ratios are polynomially bounded.  If
\(S=(k+1)N+r\), then
\[
  \frac{A(N,r+1)}{A(N,r)}
  =
  \frac{r+1}{r}\,
  \frac{(k+1)N+r}{kN+r+1}
  \le Cm,
\]
and the reciprocal ratio is also at most \(Cm\).  Similarly,
\[
  \frac{A(N+1,r)}{A(N,r)}
  =
  \frac{S}{S+k+1}\,
  \frac{\prod_{u=1}^{k+1}(S+u)}
       {(N+1)\prod_{u=1}^{k}(kN+r+u)}
  \le Cm,
\]
and the reverse adjacent ratio is bounded by the same type of estimate.  Here
all ratios are taken only between nonzero coefficients, and we use only that
all relevant indices are between \(1\) and \(Cm\), while \(k\) is fixed.  The
finitely many boundary cases with \(N\) or \(r\) bounded by a constant depending
only on \(k\) are absorbed by increasing \(C_0\); away from these cases, each
displayed factor and its reciprocal is bounded by a fixed power of \(m\).

Starting from
\[
  F_{k,m}
  =
  \frac1{km+1}\binom{(k+1)m}{m},
\]
which is \(A(m,1)\), one reaches \(A(m+h,r)\) by
\(O(|h|+r)\le Cp\) adjacent changes.  Hence
\[
  [t^{m+h}]M(t)^r\le F_{k,m}(C_0m)^{C_0p}.
\]
Multiplying by the boundary-mark factor and increasing \(C_0\) proves the
estimate.
\end{proof}

\begin{lemma}[Deterministic boundary-arc assignment]
\label{lem:arc-assignment}
Fix a scheme record and the recorded attachment positions in the cyclic word.
Then every complementary arc between consecutive recorded positions is assigned
uniquely either to a sector incident to a scheme vertex or to one side of a
specified scheme edge.  In particular, expansions of distinct scheme edges
carry no extra interleaving data beyond the recorded boundary tour.
\end{lemma}

\begin{proof}
The scheme record includes the rooted boundary tour of the one-face scheme.
This tour is a cyclic list of the scheme corners and scheme half-edges as they
are encountered along the unique face.  Once the attachment positions of these
listed objects are fixed in the original cyclic word, each complementary arc is
the interval between two consecutive listed objects.  If the two endpoints are
successive corners around a scheme vertex, the arc is the corresponding vertex
sector.  If they are the two consecutive boundary appearances on one side of a
scheme edge, the arc is that side of the edge expansion; the scheme-edge
pairing and finite endpoint/side convention specify the matching side-arc.
Any different assignment would change either the rooted boundary tour or the
recorded side pairing, hence would be a different scheme record.
\end{proof}

\begin{lemma}[Expansion counting inside a fixed scheme]
\label{lem:scheme-expansion}
Fix a rooted signed scheme with \(p\) half-edges and corners, including the
finite endpoint/side convention attached to each scheme edge, and fix the
\(O(p)\) occurrence positions of the original word at which the scheme is
attached.  Then
the number of compatible expansions of the scheme edges,
together with all rooted tree branches attached to those expansions and to the
scheme vertices, is at most
\[
  F_{k,m}(C_0m)^{C p},
\]
where \(C_0\) and \(C\) depend only on \(k\).
\end{lemma}

\begin{proof}
The following global encoding keeps the base count \(F_{k,m}\).

\emph{Boundary normalization.}
Cut the fixed cyclic word \((z^kz^{*k})^m\) at all recorded scheme attachment
positions and scheme corners.  If one of these cuts falls inside a block
\(z^k\) or \(z^{*k}\), record its offset in the block and the sign of the
block.  This creates only \(O(p)\) additional boundary marks and contributes at
most \((2k)^{Cp}\) choices.  After these finite offsets are recorded, the
remaining arcs are read as standard block-compatible arcs.  The only effect is
an \(O(p)\) shift in the number of complete \(z^kz^{*k}\)-blocks, since each
cut affects only the block in which it lies.  If several cuts fall in the same
block, their relative order is already fixed by the occurrence positions.
Thus the partial-block information is recorded only at the \(O(p)\) cuts; no
labels in the interiors of long arcs are chosen independently.

\emph{Scheme arc assignment.}
By Lemma~\ref{lem:arc-assignment}, the fixed scheme record assigns each cut arc
either to a vertex sector or to one side of a specified scheme edge.  The arcs
assigned to distinct scheme edges are therefore disjoint ordered intervals and
appear as independent ordered factors in the forest product.  There is no
further interleaving choice to record.

\emph{Edge envelopes and sector branches.}
Given a compatible expansion, encode every vertex sector by the rooted leading
tree-like branch filling that sector.  This is justified by
Lemma~\ref{lem:pruned-leading-branch}: pruned components are bridge trees, so
their one-face contour crosses every edge once in each direction.  For every
scheme edge, apply the normalization map of
Lemma~\ref{lem:edge-expansion-injection} to the two side-arcs of its
degree-two chain, together with the branches attached to that chain.  The
output is a decorated two-rooted leading envelope; the finite endpoint
convention has already been included in the scheme record.  This envelope is a
counting object, not a claim that the original signed chain is itself leading.
The original same-sign or opposite-sign status will be read after the labels
are restored in the fixed word.

\emph{Forest product and injectivity.}
Listing these sector branches and scheme-edge envelopes in the cyclic order
specified by the scheme gives an ordered forest of leading tree-like objects
with at most \(Cp\) roots and at most \(Cp\) marked corners.  The marked
corners consist of the roots of the sector branches, the two roots of each
two-rooted envelope, and the finitely many endpoint or cut offset marks; they
do not include the individual edges of a long degree-two chain.  Equivalently,
the forest stores only the \(O(p)\) cut data; the lengths and nested branches of
the chains are encoded by the block counts and leading-tree shapes inside the
resulting intervals.

This encoding is injective.  From the ordered forest, the recorded attachment
positions, and the partial-block offsets, place the roots back into the cut
arcs of the fixed word and read every unmarked interval in cyclic order.  This
restores all occurrence labels on every sector and on the two side strings
\(I,J\) of each expanded scheme edge.  Then
Lemma~\ref{lem:edge-expansion-injection} recovers the edge expansion by
pairing
\[
  i_r\longleftrightarrow j_{\theta(r)}
\]
for the appropriate convention \(\theta\).  The same-sign or opposite-sign
status of each restored pair is read from the fixed signs of the restored
occurrence labels.  Hence no variable-length twist sequence, and no
degree-two-chain label list, is part of the record; this is the content of
Lemma~\ref{lem:no-twist-record}.

\emph{Coefficient bound.}
Let \(M(t)\) be the rooted leading-tree generating function from
Lemma~\ref{lem:leading}, \(M(t)=1+tM(t)^{k+1}\).  A two-rooted envelope is
counted by choosing one of its distinguished corners as the root of the
underlying leading tree and treating the other distinguished corner as a
marked corner; marking both distinguished corners would only add another
polynomial factor.  Thus an ordered forest with \(r\) rooted leading
components is counted by \(M(t)^r\).  If this forest has \(N\) complete blocks,
then it has at most \(C(N+r)\) boundary corners; allowing one extra
corner, the number of ordered choices of \(j\) marked boundary corners is at
most \((C(N+r+1))^j\).  This covers empty or root-only components, where the
formal operator \((t\,d/dt)^j\) would be too small.  The constants also absorb
the \(2k\) possible offsets inside a block.  Since the scheme has \(O(p)\)
sectors and edge sides, we have
\[
  r,j\le Cp .
\]
Also \(p\le Cm\), because the scheme half-edges and corners are represented
by positions in the original face tour of length \(2L=2km\).  The \(O(p)\)
partial-block corrections replace the coefficient \([t^m]\) by
\([t^{m+h}]\) with \(|h|\le Cp\).  Therefore the number of possible forest
envelopes is bounded by a constant times
\[
  \sum_{\substack{0\le r,j\le Cp\\ |h|\le Cp}}
  \bigl(C(m+|h|+r+1)\bigr)^j [t^{m+h}]M(t)^r .
\]
By Lemma~\ref{lem:coefficient-comparison}, each summand is at most
\[
  F_{k,m}(C_0m)^{C_0p}.
\]
The number of triples \((r,j,h)\) is polynomial in \(p\).  Since
\(p\le Cm\), this polynomial factor is absorbed into
\((C_0m)^{Cp}\) after increasing \(C_0,C\); the case \(p=0\) is the leading
tree case.  Enlarging the constants once more absorbs the
finite scheme-edge conventions and the partial-block offset records.
\end{proof}

\begin{proof}[Proof of Lemma~\ref{lem:opening}]
If \(\chi({\mathcal P})=0\), the map has \(L\) edges and \(L+1\) vertices and
is therefore a tree.  In a tree, removing any edge separates the contour into
two components, so the two traversals of that edge must cross the cut in
opposite directions in the underlying graph.  Hence every pair joins a plus
occurrence to a minus occurrence.  The one-face tree contour is planar and compatible with the block
word \((z^kz^{*k})^m\), and is counted by Lemma~\ref{lem:leading}.

Assume now that \(s=\chi({\mathcal P})\ge1\), and let
\(G({\mathcal P})\) be the connected multigraph obtained from the quotient
polygon: its vertices are the blocks of \(\pi({\mathcal P})\), and its edges
are the pairs of occurrences.  Then
\[
  |E(G({\mathcal P}))|-|V(G({\mathcal P}))|+1=s .
\]
We use the following canonical core of this embedded graph.  First delete
recursively every degree-one edge; this removes only tree branches attached to
the cyclic part of the graph.  Then suppress every maximal path whose internal
vertices have degree two.  In the exceptional case where the remaining
two-core is a single cycle, replace it by one vertex with one loop.  The
result is a rooted one-face \emph{scheme} \(K({\mathcal P})\).  It inherits
the cyclic order of half-edges from the original face tour.  Its root is chosen
canonically as follows: start from the original root corner and move along the
boundary tour until the first retained core half-edge is met.  If the core is
the single-cycle exception, the resulting loop is rooted at that first
encountered core half-edge.  If the original root lies in a pruned branch, this
same rule records the attachment sector at which the branch rejoins the core.
No additional root-position datum inside the pruned branch is needed: the
absolute occurrence labels in the fixed boundary word are retained, and the
interval containing the original root is read back in that fixed cyclic order.
The suppressed degree-two paths are not discarded and are not treated as tree
branches.  They are the edge-expansions of the scheme.
Loops and multiple edges are allowed throughout; a loop is counted with its two
scheme half-edges, possibly incident to the same scheme vertex.
The single-cycle exception is treated by the same scheme record: the retained
loop is a scheme edge with two side strings and one finite endpoint/side
convention, and Lemma~\ref{lem:scheme-expansion} counts its full expansion as a
two-rooted edge envelope.  Thus the edges of the original cycle are not
mistaken for pruned bridges and are not forced to have opposite signs.

The scheme has size \(O(s)\).  Indeed, after the single-cycle exception has
been dealt with, every vertex of the scheme has degree at least three.  If the
scheme has \(e_K\) edges and \(v_K\) vertices, then
\[
  e_K-v_K+1=s,\qquad 2e_K\ge3v_K .
\]
Hence \(v_K\le2s-2\) and \(e_K\le3s-3\).  Enlarging the constant covers the
single-loop case, so the number \(p\) of scheme half-edges and scheme corners
is at most \(C s\).

The injective encoding starts with the abstract rooted
signed scheme: its \(p\le Cs\) labelled half-edges, their pairing into scheme
edges, the cyclic order around each scheme vertex, the rooted boundary tour,
and the finite endpoint/side convention attached to each scheme edge.  These
are the data needed later to assign complementary boundary arcs to vertex
sectors and scheme-edge sides.  The number of abstract rooted one-face schemes
of size \(p\), including these finite side conventions, is at most
\((Cp)^{Cp}\), hence at most \((Cs)^{Cs}\), which is bounded by \((CL)^{Cs}\).
Next record the occurrence labels in the original face tour at which these
scheme half-edges and scheme corners sit.  There are at most \(Cs\) such
labels, so there are at most \((2L)^{Cs}\) possibilities.

The remaining data are the expansions of the scheme edges and the tree
branches attached to them and to the scheme vertices.  Lemma~\ref{lem:scheme-expansion}
applies with \(p\le Cs\).  Here the lemma counts a degree-two chain as a
two-rooted scheme-edge network, not as a forest.  Thus
long same-sign twisted chains, such as a twisted cycle, are included in the
expansion data.  The lemma gives at most \(F_{k,m}(Cm)^{Cs}\) possibilities
for these expansions.  Combining this with the estimates for the abstract
scheme and the marked occurrence labels, and using \(L=km\), we obtain
\[
  \#\{{\mathcal P}:\chi({\mathcal P})=s\}
  \le
  F_{k,m}(CL)^{Cs}(Cm)^{Cs}
  \le
  F_{k,m}(C_1m)^{D_1s},
\]
where \(C_1,D_1\) depend only on \(k\).  This also covers \(s=0\), after
increasing the constants if necessary.
\end{proof}

\begin{proof}[Proof of Lemma~\ref{lem:surplus}]
We work only at the level of occurrence pairings and corner partitions.  This
avoids the ambiguity that would arise if one tried to modify a single edge of
an actual closed word while keeping its adjacent edges fixed.

For each ordered variable-edge class of \(\mathbf i\), choose two principal
occurrences: if the class contains both signs, choose the first plus occurrence
and the first minus occurrence; otherwise choose its first two occurrences.
Pair these two principal occurrences.  All remaining occurrences are called
surplus.  Their number is
\[
  2L-2q=2a.
\]
Pair the surplus positions in increasing cyclic order.  This is a canonical
choice, so it introduces no additional choices in the encoding.  Together with
the principal pairs this gives a pair skeleton \({\mathcal P}\).

Let \(\pi_{\mathbf i}\) be the vertex partition of the original word.  The
principal pairs impose endpoint equalities already satisfied by
\(\pi_{\mathbf i}\).  Each surplus pair may impose at most two additional
endpoint identifications.  If we add these identifications to
\(\pi_{\mathbf i}\), we obtain a partition \(\pi'\) with at least
\(v(\mathbf i)-2a\) blocks and satisfying all pair equalities of
\({\mathcal P}\).  Since \(\pi({\mathcal P})\) is the finest partition with
this property,
\[
  v({\mathcal P})\ge |\pi'|\ge v(\mathbf i)-2a.
\]
Therefore
\[
  \chi({\mathcal P})
  =L+1-v({\mathcal P})
  \le L+1-v(\mathbf i)+2a
  =b+2a .
\]

To count the information lost in passing from \(\mathbf i\) to
\({\mathcal P}\), record the \(2a\) surplus positions, and for each of
them the principal pair, among at most \(L\) choices, to whose original
ordered variable-edge it belongs.  These records recover the original
edge-occurrence partition.  In this recovery step we do not keep the artificial
endpoint identifications created by pairing the surplus positions in
\({\mathcal P}\); those artificial pairs served only to produce a skeleton
whose defect is controlled.

Let \(\pi_0\) be the finest partition generated by these original
edge-occurrence equalities.  The graph associated with \(\pi_0\) is connected:
it is obtained from the cyclic word \(0,1,\ldots,2L=0\) by quotienting corners
and edge occurrences, and a quotient of a connected closed walk is connected.
Loops are allowed; a loop is simply an ordered edge whose tail and head lie in
the same \(\pi_0\)-block, and it requires no special record beyond the same
endpoint data.  The graph has \(q=L-a\) ordered edge classes, hence
\(|\pi_0|\le q+1\).  The original vertex partition is a coarsening of
\(\pi_0\), so it is recovered by gluing at most
\[
  |\pi_0|-v(\mathbf i)\le q+1-v(\mathbf i)
  =(L-a)+1-(L+1-b)=b-a\le b
\]
additional pairs of vertices.  The last inequality is just the admissibility
bound \(v(\mathbf i)\le q+1\), equivalently \(b\ge a\).  Each gluing is
specified by two corner positions.  To make the encoding injective, we choose
these gluing records canonically: order the \(\pi_0\)-blocks by the least
corner they contain, and order the blocks of the original vertex partition by
the least corner contained in them.  Lexicographic order below refers to these
two total orders.  In each block of the original partition,
look at the induced complete graph on the \(\pi_0\)-blocks contained in it and
record the lexicographically first spanning tree of pairwise gluings between
those blocks.  The union of these spanning trees has exactly
\(|\pi_0|-v(\mathbf i)\) edges and recovers the original vertex partition from
\(\pi_0\).  The records are listed in lexicographic order of the original
partition blocks, and within each block in lexicographic order of the
spanning-tree edges.

Altogether the record list has \(O(a+b)\) entries, each with at most
\(O(L^2)\) choices.  Since \(L=km\), the number of records is bounded by
\((Cm)^{D(a+b)}\).
\end{proof}

\begin{proof}[Proof of Proposition~\ref{prop:defect-counting}]
Let \(a=L-q\) and \(b=L+1-v\), so \(r=a+b\).  By
Lemma~\ref{lem:surplus}, the word is encoded by a pair skeleton
\({\mathcal P}\) with \(\chi({\mathcal P})\le b+2a\le2r\), together with at
most \((Cm)^{Dr}\) possible local records.  Lemma~\ref{lem:opening}, summed
over \(0\le s\le2r\), gives at most
\[
  F_{k,m}\sum_{s=0}^{2r}(Cm)^{D_1s}
  \le
  F_{k,m}(Cm)^{D'r}
\]
possibilities for \({\mathcal P}\), after increasing \(D'\).  Indeed, if
\(r\ge1\), then the factor \(2r+1\) in the sum is at most \(Cm\), since
\(r\le2L+1=O_k(m)\), and is absorbed into \((Cm)^{D'r}\).  The case \(r=0\)
is included by the leading count.  Multiplying by the record count and
increasing \(D\) proves the estimate.
\end{proof}

\section{The Proof of the Moment Estimate}
\label{app:moment}

\begin{proof}[Proof of Theorem~\ref{prop:moment}]
We prove Theorem~\ref{prop:moment} for truncated entries satisfying
\[
  |w_{ij}|\le \eps_n\sqrt n,
  \qquad
  \eps_n(\log n)^B\to0
\]
for sufficiently large \(B\).

Let \(\mathbf i\) be an admissible word.  Put
\[
  q=q(\mathbf i),\qquad v=v(\mathbf i),\qquad
  a=L-q,\qquad b=L+1-v.
\]
Then \(d(\mathbf i)=a+b\).  Since every ordered variable-edge appears at least
twice and the entries are bounded by \(\eps_n\sqrt n\), we have
\[
  \left|
  \E\prod_{s=1}^{2L}\zeta_s(\mathbf i)
  \right|
  \le
  (\eps_n\sqrt n)^{2a}
  =
  \eps_n^{2a} n^a
\]
for large \(n\).  Indeed, fix an ordered variable-edge \(e=(u,v)\), and let
\(r_e=r_e(\mathbf i)\).  The factors with \(e_s(\mathbf i)=e\) are either
copies of \(w_{uv}\) or copies of \(\ol{w_{uv}}\), according to the sign of the
occurrence.  Hence their joint contribution is bounded in absolute value by
\[
  \E|w_{uv}|^{r_e}
  \le
  (\eps_n\sqrt n)^{r_e-2}\E|w_{uv}|^2
  =
  (\eps_n\sqrt n)^{r_e-2}.
\]
Multiplying this estimate over all \(q\) distinct ordered variable-edges and
using \(\sum_e(r_e-2)=2(L-q)=2a\) gives the displayed bound.

The contribution of all words with fixed \(a,b\) is therefore bounded by
\[
  n^{-L}\,
  n^v\,
  \eps_n^{2a} n^a
  \cdot
  \#\{\text{classes with }L-q=a,\ L+1-v=b\}.
\]
Because \(v=L+1-b\), and because the classes with these fixed values of
\(a,b\) are a subset of all classes with defect \(a+b\), this is bounded by
\[
  n\,
  \eps_n^{2a}
  n^{a-b}
  \cdot
  \#\{\text{classes with defect }a+b\}.
\]
Since \(v\le q+1\), we have
\[
  b=L+1-v\ge L-q=a.
\]
Write \(c=b-a\ge0\).  Then \(a+b=2a+c\), and the preceding contribution is
\[
  n\,\eps_n^{2a}n^{-c}
  \cdot
  \#\{\text{classes with defect }2a+c\}.
\]
By Proposition~\ref{prop:defect-counting},
\[
  \#\{\text{classes with defect }2a+c\}
  \le F_{k,m}(Cm)^{D(2a+c)}.
\]
Summing over \(a,c\ge0\) gives
\begin{align*}
  \E\Tr\bigl((X^kX^{*k})^m\bigr)
  &\le
  nF_{k,m}
  \sum_{a,c\ge0}
  \left((Cm)^{2D}\eps_n^2\right)^a
  \left((Cm)^D n^{-1}\right)^c .
\end{align*}
After \(k,A\) and the constants in Proposition~\ref{prop:defect-counting} have
been fixed, choose \(B\) large enough that
\((CA\log n)^{2D}\eps_n^2\to0\).  The hypothesis
\(\eps_n(\log n)^B\to0\) gives this, for instance, for any \(B>D\) after
absorbing constants.  Then
\[
  (CA\log n)^{2D}\eps_n^2\to0.
\]
For \(1\le m\le A\log n\),
\[
  (Cm)^{2D}\eps_n^2
  \le (CA\log n)^{2D}\eps_n^2=o(1),
  \qquad
  (Cm)^Dn^{-1}
  \le (CA\log n)^Dn^{-1}=o(1).
\]
Hence both geometric sums are \(1+o(1)\), uniformly for
\(1\le m\le A\log n\).  Therefore
\[
  \E\Tr\bigl((X^kX^{*k})^m\bigr)
  \le
  nF_{k,m}(1+o(1)).
\]
Using \eqref{eq:F-growth}, choose \(n\) so large that
\(1+o(1)\le 1+\eta/\gamma_k^2\).  Since \(m\ge1\),
\[
  F_{k,m}(1+o(1))
  \le
  \gamma_k^{2m}\left(1+\frac{\eta}{\gamma_k^2}\right)
  \le
  (\gamma_k^2+\eta)^m .
\]
Thus
\[
  \E\Tr\bigl((X^kX^{*k})^m\bigr)
  \le
  n(\gamma_k^2+\eta)^m
\]
for all sufficiently large \(n\), uniformly over \(m\le A\log n\).
\end{proof}

\bibliographystyle{plain}
\bibliography{reference}

\end{document}